\documentclass[11pt,letterpaper]{amsart}

\usepackage{amsmath}
\usepackage{amsthm}
\usepackage{amsfonts}
\usepackage{amssymb}
\usepackage{comment}
\usepackage{hyperref}
\usepackage{xcolor}

\newtheorem{theorem}{Theorem}[section]
\newtheorem{lemma}[theorem]{Lemma}
\newtheorem{corollary}[theorem]{Corollary}
\newtheorem{proposition}[theorem]{Proposition}

\newcommand{\p}[1]{\noindent {\newline\bf #1.}}
\newcommand{\aut}{\operatorname{Aut}}
\newcommand{\fix}{\operatorname{Fix}}

\newcommand{\im}{\operatorname{im}}
\newcommand{\rk}{\operatorname{rk}}
\renewcommand{\top}{\operatorname{top}}
\newcommand{\graph}{\operatorname{graph}}

\newcommand{\BS}{\operatorname{BS}}
\newcommand{\onto}{\twoheadrightarrow}

\newcounter{alancomments}

\title[Conjugacy Problem for certain HNN-extensions]{The Conjugacy Problem for ascending HNN-extensions of free groups}
\author{Alan D. Logan}%
\date{\today}
\address{Heriot--Watt University\\ Edinburgh, EH14 4AS, UK.}
\email{Alan.Logan@hw.ac.uk}


\begin{document}
\maketitle

\begin{abstract}
We give an algorithm to solve the Conjugacy Problem for ascending HNN-extensions of free groups.
To do this, we give algorithms to solve certain problems on dynamics of free group endomorphisms.
\end{abstract}


\section{Introduction}
The Conjugacy Problem for a group $G=\langle \mathbf{x}\mid\mathbf{r}\rangle$ is the problem of determining whether two words $u, v$ over $\mathbf{x}^{\pm1}$ define conjugate elements of $G$, written $u\sim v$. It is one of Dehn's three fundamental problems for groups defined in terms of generators and relators, along with the Word and Isomorphism Problems, and is known to be undecidable in general \cite{Miller1992Decision}. Indeed, Wei{\ss} proved in his PhD thesis that this problem is undecidable in general for finitely presented HNN-extensions of finitely generated free groups \cite[Theorem 5.8]{weiss2015complexity}.

Let $F$ be a free group, and $\phi: F\to F$ an injective group endomorphism.
The \emph{ascending HNN-extension of $F$ induced by $\phi$} is the group with relative presentation
\[
F\ast_{\phi}:=\langle F, t\mid t^{-1}xt=\phi(x),\:x\in F\rangle.
\]

Our main theorem is as follows:

\begin{theorem}[Theorem \ref{thm:CPBODY}]
\label{thm:CP}
The Conjugacy Problem is decidable for ascending HNN-extensions of finitely generated free groups.
\end{theorem}

\p{\boldmath{$3$}-manifold groups}
Feighn and Handel drew parallels between ascending HNN-extensions of free groups and mapping tori of surfaces (i.e. fibered $3$-manifolds) \cite{Feighn1999Mapping}.
This view extends the parallels between mapping class groups and outer automorphism groups $\operatorname{Out}(F)$ of free groups \cite{Vogtmann2015geometry} to additionally incorporate injective, non-surjective maps (these only occur in the free group setting, as surface groups are co-Hopfian), and highlights similarities between these HNN-extensions and $3$-manifold groups.
Indeed, just like $3$-manifold groups, ascending HNN-extensions of free groups are coherent \cite{Feighn1999Mapping} and residually finite \cite{Borisov2005polynomial}, and they are hyperbolic if and only if they have no Baumslag--Solitar subgroups \cite{Mutanguha2021dynamics}.
In addition, they admit an analogue of the Thurston norm \cite{Funke2018Alexander}.
The Conjugacy Problem is known to be decidable for $3$-manifold groups \cite[Theorem 4.10]{Aschenbrenner2015Decision}, so Theorem \ref{thm:CP} adds to the list of properties shared by ascending HNN-extensions of free groups and $3$-manifold groups.

\p{One-relator groups}
A one-relator group is a group with presentation $\langle \mathbf{x}\mid R\rangle$ for some word $R\in F(\mathbf{x})$.
These groups are one of the cornerstones of Geometric and Combinatorial Group Theory; see for example the recent work of Louder and Wilton \cite{Louder2022negative}.
HNN-extensions underlie the theory of one-relator groups \cite{McCool1973One}, and ascending HNN-extensions of free groups are of particular importance in this theory.
For example, the first example of a non-linear residually finite one-relator group, due to Dru\c{t}u and Sapir, is an ascending HNN-extension of a free group \cite{Drutu2005Nonlinear}.
Their importance is further highlighted by genericity results:
Almost all $n$-generator $1$-relator groups with $n\geq3$ are subgroups of one-relator groups which are ascending HNN-extensions of free groups \cite{Sapir2011AlmostAll} (the fact that the ascending HNN-extensions are one-relator groups follows from the last line of the proof of their Theorem 2.5).
In the two-generator case, a non-zero proportion of all two-generator one-relator groups are ascending HNN-extensions of free groups with $\phi$ an automorphism \cite{Dunfield2006random}, and computer experimentation suggests that around 99\% of one-relator groups are subgroups of ascending HNN-extensions of free groups \cite[Footnote on p3]{Borisov2005polynomial}.

In the 1930s, Magnus gave a solution to the Word Problem for one-relator groups \cite{Magnus1932wordproblem}, but the Conjugacy Problem remains open \cite[Problem 3.34]{khukhro2014unsolved} \cite[Question O5]{Baumslag2002Open} \cite[Question OR7]{Fine2001Open}.
Much of the progress on the Conjugacy Problem here has centered on the shape of the words, such as being a proper power \cite{Newman1968some}, being cyclically or conjugacy pinched \cite{Larsen1977conjugacy}, or satisfying Pride's unique max-min condition \cite{Pride2008residual}.
Theorem \ref{thm:CP} builds on this theme, as a result of Brown uses the shape of the relator $R$ to classify when a one-relator group is an ascending HNN-extension of a free group \cite[Section 4]{Brown1987trees}.

Theorem \ref{thm:CP} applies to one-relator groups which are beyond the reach of the previous results cited above.
The aforementioned non-linear residually finite one-relator group is an example, as we explain in Section \ref{sec:oneRel}.

\p{Dynamics of endomorphisms}
An ascending HNN-extension $F\ast_{\phi}$ of a free group is \emph{free-by-cyclic} if the endomorphism $\phi$ is additionally an automorphism (the class of groups often referred to as ``free-by-cyclic'' additionally includes some virtually-free groups; we ignore these so as to simplify our exposition).
The Conjugacy Problem is known to be decidable for free-by-cyclic groups, and we are aware of two proofs of this.

One proof uses Dehn functions with low complexity:
If an HNN-extension of a free group has Dehn function less than $n^2 \log n$ (with a somewhat technical definition of ``less than'') then it has decidable Conjugacy Problem \cite[Theorem 2.5]{olshanskii2006SmallDehn}. Free-by-cyclic groups have quadratic Dehn function \cite{Bridson2010isoperimetric}, and so have decidable Conjugacy Problem.
However, this proof is not adaptable to the current setting as, for example, $\BS(1, 2)=\langle a, t\mid t^{-1}at=a^2\rangle $ is an ascending HNN-extension of a (cyclic) free group, and has exponential Dehn function \cite{Groves2001Isoperimetric}.

The other proof uses dynamics of free group automorphisms \cite{Bogopolski2006conjugacy}, and is built around two key algorithms:
An algorithm of Bogopolski and Maslakova which computes a basis of the ``fixed subgroup'' of an automorphism \cite{Bogopolski2016algorithm}, and an algorithm of Brinkmann, which we refer to as \emph{Brinkmann's conjugacy algorithm} and which determines for elements $u, v\in F$ and an automorphism $\phi\in\aut(F)$ if there exists some exponent $p\in\mathbb{Z}$ such that $\phi^p(u)\sim v$ \cite[Theorem 6.2]{Brinkmann2010Detecting}.
Both algorithms apply train-track maps.

The proof of Theorem \ref{thm:CP}
follows roughly the same route as this dynamics proof but for injective, non-surjective endomorphisms.
Indeed, the proof also relies on two key algorithms:
An algorithm of Mutanguha which generalises the algorithm of Bogopolski and Maslakova to endomorphisms, while our key technical result is an appropriate analogue of Brinkmann's conjugacy algorithm for injective endomorphisms; see Proposition \ref{prop:PowSuperTwCP}.
Both algorithms apply automorphic expansions,
which replace train-track maps as the central topological objects in the theory of non-surjective endomorphisms of free groups.

With a bit of effort, these two algorithms give a general result on the dynamics of injective endomorphisms (Proposition \ref{prop:UltraTwCPINJECTIVE}), from which Theorem \ref{thm:CP} follows quickly.
Proposition \ref{prop:UltraTwCPINJECTIVE} generalises easily to all endomorphisms, so we state this version now although it is beyond what is required for Theorem \ref{thm:CP}.
Two elements $u, v\in F$ are \emph{$\varphi$-twisted conjugate} if there exists some $x\in F$ such that $u=x^{-1}v\varphi(x)$.
We write $\mathbb{N}$ for the non-zero natural numbers, and $\mathbb{N}_0$ for $\mathbb{N}\cup\{0\}$.

The following is completely general, and supersedes the analogous result implicit in \cite{Bogopolski2006conjugacy} for automorphisms (for the automorphism case, $p$ may be assumed to be $0$).

\begin{theorem}[Theorem \ref{thm:UltraTwCPBODY}]
\label{thm:UltraTwCP}
There exists an algorithm with input an endomorphism $\phi$ of a free group $F$, a natural number $n\in\mathbb{N}_0$ and two elements $u, v\in F$, and output either a pair $(p, q)\in \mathbb{N}_0\times\mathbb{N}_0$
such that $\phi^p(u)$ is $\phi^n$-twisted conjugate to $\phi^q(v)$, or $\mathrm{no}$ if no such pair exists.
\end{theorem}

We prove Theorem \ref{thm:UltraTwCP} in Section \ref{sec:MainDynamicsProofs}, which is a brief interlude in our proof of Theorem \ref{thm:CP}.
This section extends results of the previous sections to non-injective endomorphisms, and develops them to deal with equality rather than ($\phi^n$-twisted) conjugacy.
For example, we prove the following result on equality.
This generalises a second algorithm of Brinkmann, which we refer to as \emph{Brinkmann's equality algorithm} \cite[Theorem 6.3]{Brinkmann2010Detecting}.

\begin{theorem}[Theorem \ref{thm:equalityBODY}]
\label{thm:equality}
There exists an algorithm with input an endomorphism $\phi$ of a free group $F$ and two elements $u, v\in F$, and output either a pair $(p, q)\in \mathbb{N}_0\times\mathbb{N}_0$
such that $\phi^p(u)=\phi^q(v)$, or $\mathrm{no}$ if no such pair exists.
\end{theorem}

\p{Outline of the paper}
In Section \ref{sec:Twisted} we give an algorithm to solve the Twisted Conjugacy Problem for free groups, and we prove Theorem \ref{thm:UltraTwCP} for $n>1$ (Proposition \ref{prop:ExtendedTwCP}).
Sections \ref{sec:iterates} and \ref{sec:prelim} are the technical parts of the paper, proving two generalisations of Brinkmann's conjugacy algorithm for $\phi$ injective.
In Section \ref{sec:MainDynamicsProofs} we extend the aforementioned generalisations of Brinkmann's conjugacy algorithm from injective to non-injective maps, and we prove Theorems \ref{thm:UltraTwCP} and \ref{thm:equality}.
In Section \ref{sec:MainProofs} we prove Theorem \ref{thm:CP}.
In Section \ref{sec:oneRel} we note that Theorem \ref{thm:CP} applies to the non-linear residually finite one-relator group of Dru\c{t}u and Sapir, but that previous results on the Conjugacy Problem do not apply to this group.

\p{Acknowledgements}
I would like to thank Laura Ciobanu for her support and help throughout this project, and
Jean-Pierre Mutanguha for helpful conversations about automorphic expansions, for his careful reading and comments on Sections \ref{sec:iterates} \& \ref{sec:prelim}, and for explaining his proof of Proposition \ref{prop:JPMremark} to me.
I am also grateful to Andr\'e Carvalho for discussions which led to Section \ref{sec:MainDynamicsProofs},
to Jack Button for pointing out work of Levitt which gives Proposition \ref{prop:Levitt},
to Matt Zaremsky for pointing me to the Rosenberger citation relating to Lemma \ref{lem:primitivity},
and to Tara Brendle, Giles Gardam, Marco Linton and Henry Wilton for sharing their opinions and knowledge on motivations and one-relator groups.
Finally, I am grateful to the anonymous referee for a number of comments which helped improve the paper, as well as for pointing out an issue with the original proof of Lemma \ref{lem:primitivity}.
This research was supported by EPSRC grants EP/R035814/1 and EP/S010963/1.


\section{Twisted conjugacy}
\label{sec:Twisted}
In this section we prove Theorem \ref{thm:UltraTwCP} for $n>0$, and to do this we consider a problem of general interest:
The \emph{Twisted Conjugacy Problem} for free groups is the problem of determining for any endomorphism $\phi:F\to F$ and elements $u, v\in F$, if there exists some $x\in F$ such that $u=x^{-1}v\phi(x)$.
This problem has many applications in topology and Nielsen fixed point theory \cite{Jiang2005Primer}.
A solution to the automorphism-Twisted Conjugacy Problem for free groups (i.e. when $\phi:F\to F$ is an automorphism) was the main original step in the paper solving the Conjugacy Problem for free-by-cyclic groups \cite[Theorem 1.5]{Bogopolski2006conjugacy}.

Note that we could replace the identity ``$u=x^{-1}v\phi(x)$'' in the above definition with ``$u=\phi(x)^{-1}vx$''. This yields two distinct definitions of twisted conjugate, yielding a ``Right'' and ``Left'' Twisted Conjugacy Problem. However, as $u=x^{-1}v\phi(x)$ if and only if $u^{-1}=\phi(x)^{-1}v^{-1}x$, these problems are equivalent. We therefore focus only on the identity $u=x^{-1}v\phi(x)$.

We now establish a solution to the Twisted Conjugacy Problem for free groups, based on the following result of Mutanguha.
In Section \ref{sec:prelim} we dig into the proof of this result, and in particular state and apply the theorem from which it follows (Theorem \ref{corol:fixedAlg}).

\begin{corollary}[Mutanguha, \cite{JPM}]
\label{corol:fixedAlg}
There exists an algorithm with input a finitely generated free group $F$ and an endomorphism $\phi:F\to F$ and output a basis for the fixed subgroup $\fix(\phi)$ of $\phi$.
\end{corollary}

Ventura recently gave a different proof of our Lemma \ref{lem:TwCP}, also based on Corollary \ref{corol:fixedAlg} \cite{ventura2021multiple}.
We include our short proof for completeness.
We use $\varphi$ rather than $\phi$ in Lemma \ref{lem:TwCP} to align with the notation in the proof of Proposition \ref{prop:ExtendedTwCP}.

\begin{lemma}
\label{lem:TwCP}
There exists an algorithm with input an endomorphism $\varphi$ of a free group $F$ and two elements $u, v\in F$, and which determines if $u$ and $v$ are $\varphi$-twisted conjugate.
\end{lemma}

\begin{proof}
Suppose that $F$ is the free group over the alphabet $\Sigma$, and write $F'$ for the free group over the alphabet $\Sigma\cup\{B, E\}$. Define the endomorphism $\varphi':F'\to F'$ by $\varphi'(w):=\varphi(w)$ for all $w\in F$, $\varphi'(B):=Bv$ and $\varphi'(E):=u^{-1}E$. Then there exists $x\in F$ such that $u=x^{-1}v\varphi(x)$ if and only if there exists an element $y\in\fix(\varphi')$ which has the form $y=BzE$ for $z\in F$ (in fact, $z=x$). By Corollary \ref{corol:fixedAlg}, we can compute a basis for $\fix(\varphi')$, and hence (using for example $F$-labelled graphs \cite{Kapovich2002Stallings}) we can determine if such an element $y\in\fix(\varphi')$ exists. The result follows.
\end{proof}

It is an elementary exercise to prove that twisted conjugacy classes, i.e. those sets $[u]_{\phi}:=\{v\in F\mid \:\exists\:x\in F \:s.t.\:u=x^{-1}v\phi(x)\}$, are equivalence classes.
More subtle is the following useful result which partitions these equivalence classes.

\begin{lemma}
\label{lem:TwCC}
Let $\phi$ be an endomorphism of a free group $F$. Then for all $i, j\in\mathbb{N}_0$ and all $u\in F$, $\phi^i(u)$ is $\phi$-twisted conjugate to $\phi^j(u)$.
\end{lemma}

\begin{proof}
As $v=v\phi(v)\phi(v)^{-1}$, we have $[v]_{v}=[\phi(v)]_{\phi}$ for all $v\in F$.
In particular, $[\phi^i(u)]_{\phi}=[\phi^{i+1}(u)]_{\phi}$ for all $u\in F$ and $i\geq0$.
By transitivity of twisted conjugacy classes, it follows that $[u]_{\phi}=[\phi^i(u)]_{\phi}$ for all $u\in F$ and $i\geq0$.
Another application of transitivity gives us that $[\phi^i(u)]_{\phi}=[\phi^j(u)]_{\phi}$ for all $u\in F$ and $i,j\in\mathbb{N}_0$, as required.
\end{proof}

We now prove Theorem \ref{thm:UltraTwCP} for $n>0$.
The following is completely general, and supersedes the analogous result in \cite[Proof of Proposition 1.4]{Bogopolski2006conjugacy} for automorphisms (for the automorphism case, $p$ may be assumed to be $0$).

\begin{proposition}
\label{prop:ExtendedTwCP}
There exists an algorithm with input an endomorphism $\phi$ of a free group $F$, a non-zero natural number $n\in\mathbb{N}$ and two elements $u, v\in F$, and output either a pair $(p, q)\in \mathbb{N}_0\times\mathbb{N}_0$ such that $\phi^p(u)$ is $\phi^n$-twisted conjugate to $\phi^q(v)$, or $\mathrm{no}$ if no such pair exists.
\end{proposition}

\begin{proof}
Any $p\in\mathbb{N}_0$ decomposes as $np'+r_p$ for some $p'\in\mathbb{N}_0$ and some $0\leq r_p<n$, and by Lemma \ref{lem:TwCC} we have that, $\phi^{r_p}(u)$ is $\phi^n$-twisted conjugate to $\phi^p(u)$.
Similarly, for all $q\in\mathbb{N}_0$ there exists some $0\leq r_q<n$ such that $\phi^{r_q}(v)$ is $\phi^n$-twisted conjugate to $\phi^q(v)$.
Therefore, there exists a pair $(p, q)\in\mathbb{N}_0\times\mathbb{N}_0$ such that $\phi^p(u)$ is $\phi^n$-twisted conjugate to $\phi^q(v)$ if and only if there exists such a pair $(r_p, r_q)\in\mathbb{N}_0\times\mathbb{N}_0$ which additionally satisfies $r_p, r_q<n$.

The required algorithm is therefore to apply the algorithm of Lemma \ref{lem:TwCP} with $\varphi:=\phi^n$ to all pairs $\phi^i(u), \phi^j(v)$ with $i, j<n$.
\end{proof}


\section{Stable iterates}
\label{sec:iterates}
In the following two sections we prove the generalisation of Brinkmann's conjugacy algorithm required for Theorem \ref{thm:UltraTwCP}, in the case of $\phi$ injective.

Our first task is to understand how large powers of $\phi$ behave, and to do this we prove a stability result. Specifically, we consider elements of $\phi^i(F)$ which are conjugate in $F$, and we prove that there exists a {computable} number $k\in\mathbb{N}$ such that for all $i\geq k$ any ``new'' conjugations which appear when progressing from $\phi^i(F)$ to $\phi^{i+1}(F)$ are of the form $\phi^{i+1}(c)\sim c^d$ for $d\geq2$.
This is encoded in the graphical notion of the ``stable iterate'' $\hat{\Lambda}_k$ of $\phi$, defined below in Section \ref{sec:Stabilisation}.
We then prove that we can tell if a given element stabilises under powers of $\phi$ in this way; formally, if $\phi^i(w)$ is carried by the stable iterate $\hat{\Lambda}_k$.

\subsection{Graphs, graphs of roses, and automorphic expansions}
We begin with some important definitions, starting in the standard setting of maps of graphs.
Stallings introduced this theory in the 1980s \cite{stallings1983topology}, and it is an extremely powerful tool for the study of subgroups of free groups \cite{Kapovich2002Stallings}. We then define maps of ``graphs of roses''; automorphic expansions are a specific kind of map of graphs of roses.

\p{Topological graphs}
For brevity, we shall write ``component'' for ``connected component'' throughout.
A \emph{based topological graph $(\Gamma, \star)$} is a finite $1$-dimensional CW-complex $\Gamma$ along with a distinguished vertex ($0$-cell), called a \emph{basepoint}, for each component of $\Gamma$.
Such a graph $(\Gamma, \star)$ is \emph{connected} if $\Gamma$ is connected, and \emph{degenerate} if it is a finite set of points.
For convenience, we often shorten ``based topological graph'' to ``topological graph'', and write $\Gamma$ rather than $(\Gamma, \star)$.
The \emph{core graph} of a based topological graph $\Gamma$ is the smallest subgraph which contains all based loops in each component of $\Gamma$ (this can be obtained by iteratively removing all degree-one vertices, except base points).

\p{Maps of graphs}
A \emph{based map of graphs $f:\Gamma\to\Gamma'$} is a continuous function of based topological graphs that sends vertices to vertices, base points to base points, and edges to (possibly degenerate) paths.
When the context is clear, we refer to $f$ as a \emph{map}.
Such a map $f$ is \emph{folded} if it is a topological immersion (i.e. locally injective).
The \emph{fundamental groupoid $\pi_1(\Gamma)$} of $\Gamma$ is the set of based, topologically immersed loops of $\Gamma$, under the operation of concatenation and tightening.
A map of graphs $f:\Gamma\to\Gamma'$ induces a groupoid homomorphism $f_*:\pi_1(\Gamma)\to\pi_1(\Gamma')$, and decomposes uniquely as $h\circ g$ where $g:\Gamma\to\widehat{\Gamma}$ is $\pi_1$-surjective and $h:\widehat{\Gamma}\to\Gamma'$ is folded (and hence $\pi_1$-injective) \cite[Section 3.3 \& Corollary 4.4 \& Proposition 5.3]{stallings1983topology}.

\p{$\mathbf{F}$-labelled graphs}
We can associate the free group $F$ with basis $a_1, a_2, \ldots, a_k$ to a rose $R_F$ with $k$ petals by choosing an orientation of the edges of $R_F$ and identifying the $i^{\operatorname{th}}$ petal with the basis element $a_i$; this induces an isomorphism $\mu: F\to \pi_1(R_F)$.
If $f:\Gamma\to R_F$ is a folded map of graphs then the labelling of $R_F$ induces an \emph{$F$-labelling} on $\Gamma$, and labels of based loops in $\Gamma$ correspond precisely to the subgroup $f_*(\pi_1(\Gamma))\leq\pi_1(R_F)$ (this is the main observation in the combinatorial theory of maps of graphs, detailed in \cite{Kapovich2002Stallings}).

Any endomorphism $\varphi: F\to F$ is equivalent to a map of graphs $\overline{\varphi}: R_F\to R_F$ such that $\overline{\varphi}_*\circ\mu=\mu\circ\varphi$.
This folds to give a map $\widehat{\varphi}: \widehat{R_F}\to R_F$.
Fixing $\phi:F\to F$, we shall write $\widehat{\phi^i}: \Gamma_i\to R_F$ for the folded map of graphs corresponding to $\phi^i$.
We shall write $\Lambda_i$ for the non-contractible components of the core graph of the pullback of $\widehat{\phi^i}$ with itself.
This section deals with the behaviour of $\Lambda_i$ as $i$ increases.

We now turn to automorphic expansions, which give a more refined view of the maps $\overline{\phi^i}: R_F\to R_F$.

\p{Graphs of roses}
A \emph{rose} is a topological graph with exactly one vertex.
A \emph{(connected) graph of roses} is a triple $(\Gamma, \mathcal{G}, g)$ where $\Gamma$ is a (connected) topological graph, $\mathcal{G}$ is a disjoint union of roses (so also a topological graph), and $g: \mathcal{G}\to\Gamma^{(0)}$ is a map of CW-complexes which is bijective on $0$-cells, that is, $g$ associates to each vertex $v$ of $\Gamma$ a rose $g^{-1}(v)$ from $\mathcal{G}$.
We call the topological graph $\Gamma$ the \emph{underlying subcomplex}, and the components of $\mathcal{G}$ the \emph{vertex roses}.
Following Mutanguha, we often write $(\Gamma, \mathcal{G})$ for a graph of roses, so suppress the map $g$ from the notation.
A \emph{vertex} of $(\Gamma, \mathcal{G}, g)$ is a vertex of $\Gamma$, while a \emph{subgraph} is a triple $(\Delta, g^{-1}(\Delta^{(0)}), g|_{g^{-1}(\Delta^{(0)})})$ where $\Delta$ is a subgraph of $\Gamma$ (the map in this triple is simply the obvious restriction map).
A \emph{component} is a subgraph where additionally $\Delta$ is a component of $\Gamma$.

If $\Gamma$ is a topological graph then the \emph{induced graph of roses $\graph(\Gamma)$} is the graph of roses formed by labelling every vertex of $\Gamma$ with a singleton.

\p{Maps of graphs of roses}
A \emph{based map of graphs of roses $(f, \Psi): (\Gamma, \mathcal{G}, g)\to(\Gamma', \mathcal{G}', g')$} is the following data:
\begin{enumerate}
\item a pair of maps of graphs $f:\Gamma\to\Gamma'$ and $\Psi: \mathcal{G}\to \mathcal{G}'$ satisfying $g'\circ \Psi=f\circ g$; and
\item\label{MoGoR:2} an assignment to each (oriented) edge $e$ of $\Gamma$, a sequence of based loops $\rho(e)_i \in\pi_1(g'^{-1}({v_i}))$ where $(v_i)_{i\geq0}$ is the sequence of vertices along the edge-path $f(e)$.
\end{enumerate}
When the context is clear, we refer to $(f, \Psi)$ as a \emph{map}.

If $f:\Gamma\to\Gamma'$ is a map of graphs, then this naturally induces a map of graphs of roses $\graph(f):\graph(\Gamma)\to\graph(\Gamma')$.

Each graph of roses $(\Gamma, \mathcal{G})$ can be naturally associated to a topological graph $\top(\Gamma, \mathcal{G})$, called the \emph{blow-up} of $(\Gamma, \mathcal{G})$, and this is formed by identifying each vertex $v\in\Gamma$ with the basepoint of the vertex rose $g^{-1}(v)$, and by setting the basepoints of $\top(\Gamma, \mathcal{G})$ to be the images of the basepoints of $\Gamma$.
By (\ref{MoGoR:2}) in our definition, paths in $(\Gamma, \mathcal{G})$ project to paths in the blow-up.
A map of graphs of roses $(f, \Psi)$ is \emph{$\pi_1$-injective} if $\top(f, \Psi)$ is $\pi_1$-injective.
A \emph{path $\gamma$ in $(\Gamma, \mathcal{G})$} is a sequence of paths in $\Gamma$ and $\mathcal{G}$ such that the projection $\top(\gamma)$ of $\gamma$ to $\top(\Gamma, \mathcal{G})$ is a path in $\top(\Gamma, \mathcal{G})$. Similarly, a \emph{loop in $(\Gamma, \mathcal{G})$} is a path which projects to a loop in $\top(\Gamma, \mathcal{G})$.
An \emph{immersion} of graphs of roses is a $\pi_1$-injective map of graphs of roses satisfying certain additional conditions; these conditions are rather complicated and are not applied in this paper. The only fact that we require is that the image under an immersion of an immersed loop it itself an immersed loop. We refer the reader to Mutanguha's work for a precise definition of immersions of graphs of roses \cite{JPM} (see also \cite{Mutanguha2021dynamics}, where they are called \emph{relative immersions}).
Such maps relate to topological immersions on the Bass-Serre tree, rather than to immersions of the blow-up graph.

\p{Automorphic expansions}
We are primarily concerned with maps $(f, \Psi): (\Gamma, \mathcal{G})\to(\Gamma, \mathcal{G})$ from a graph of roses to itself, and so from now on we consider just this case.

A subgraph of $(\Gamma, \mathcal{G})$ is \emph{$(f, \Psi)$-invariant} if the underlying subcomplex is $f$-invariant, i.e. $f(\Gamma)=\Gamma$.
The \emph{stable subgraph for $(f, \Psi)$} is the $(f, \Psi)$-invariant subgraph of $(\Gamma, \mathcal{G})$ consisting of $f$-periodic vertices and edges.
An \emph{expansion} is an immersion $(f, \Psi)$ whose stable subgraph is degenerate, i.e. a collection of points.
A map $(f, \Psi)$ is \emph{automorphic} if $\Psi$ restricts to a based homotopy equivalence of the $\Psi$-periodic components of $\mathcal{G}$ (i.e. to those roses $R$ with $f$-periodic associated vertex $g(R)\in\Gamma^{(0)}$).
An \emph{automorphic expansion} is therefore an expansion $(f, \Psi): (\Gamma, \mathcal{G})\to(\Gamma, \mathcal{G})$ which is automorphic.

\p{Representing homomorphisms}
A \emph{homotopy equivalence} is a map $(f, \Psi)$ whose blow-up $\top(f, \Psi)$ is a based homotopy equivalence.
Two maps $(f, \Psi): (\Gamma, \mathcal{G})\to(\Gamma, \mathcal{G})$ and $(f', \Psi'): (\Gamma', \mathcal{G}')\to(\Gamma', \mathcal{G}')$ are \emph{homotopic} via the homotopy equivalence $\alpha: (\Gamma, \mathcal{G})\to (\Gamma', \mathcal{G}')$ if $\top((f', \Psi')\circ\alpha)$ and $\top(\alpha\circ(f, \Psi))$ are homotopic relative to basepoints.

Mutanguha proved that there exists an algorithm which constructs for every injective endomorphism $\varphi:F\to F$ an automorphic expansion of a graph of roses, such that the automorphic expansion represents $\varphi$.
\begin{theorem}[Mutanguha, \cite{JPM}]
\label{thm:JPM}
There exists an algorithm with input an injective endomorphism $\varphi: F\to F$ and output a graph of roses $(\Delta, \mathcal{D})$, a homotopy equivalence $\alpha: \operatorname{graph}(R_F)\to(\Delta, \mathcal{D})$, and an automorphic expansion $\widetilde{\phi}=(f,\Psi):(\Delta, \mathcal{D})\to (\Delta, \mathcal{D})$ homotopic to $\operatorname{graph}(\overline{\varphi})$ via $\alpha$. 
\end{theorem}

\subsection{Stabilisation of iterated pullbacks}
\label{sec:Stabilisation}
Mutanguha used certain maps of graphs based around an iterated pullback construction to understand identities of the form $\phi^t(c)\sim c^d$.
We review this construction, and then prove that after finitely many steps it stabilises, and furthermore that this ``{stable iterate}'' can be computed.

\p{Iterated pullbacks}
Recall that for $\phi:F\to F$, we write $\widehat{\phi^i}: \Gamma_i\to R_F$ for the folded map of graphs corresponding to $\phi^i$, and $\Lambda_i$ for the non-contractible components of the core graph of the pullback of $\widehat{\phi^i}$ with itself.
We may post-compose the corresponding map $\Lambda_i\to R_F$ with $\overline{\phi}$ and fold it; denote the corresponding $F$-labelled graph by $\phi(\Lambda_{i})$.
This graph is a subset of $\Lambda_{i+1}$, so let $\hat\Lambda_{i+1}$ denote the complement $\Lambda_{i+1}\setminus\phi(\Lambda_{i})$.
Algebraically, each component of $\Lambda_j$ corresponds to an element of the set:
\[
\Omega_j:=\{[\phi^j(F)\cap\phi^j(F)^x]\mid x\in F\}\setminus\{[\epsilon]\}
\]
where $\epsilon$ is the identity of $F$, while each component of $\hat\Lambda_j$ corresponds to an element of:
\[
\hat{\Omega}_j:=\{[\phi^j(F)\cap\phi^j(F)^x]\mid x\in F\setminus\im(\phi)\}
\]
We refer to the graphs $\hat\Lambda_j$ as \emph{iterates}.
For $j\gg0$ the iterates ``stabilise'', and so they capture the dynamics of large powers of $\phi$.
The notion of stability we use is based on cyclic subgroup systems.

\p{Subgroup systems}
A \emph{(cyclic) subgroup system} of $F$ is a finite collection of (cyclic) subgroups $\mathcal{A} = \{A_1,\ldots , A_l\}$ of $F$.
We call the $A_i\in\mathcal{A}$ the \emph{components} of $\mathcal{A}$.
The \emph{empty system} contains no components, i.e. is empty, while the \emph{trivial system} contains only the trivial subgroup.
For subgroup systems $\mathcal{A}, \mathcal{B}$, we say $\mathcal{A}$ \emph{carries} $\mathcal{B}$ if for every component $B \in \mathcal{B}$, there is a component $A \in \mathcal{A}$ and element $x \in F$ such that $B \leq xAx^{-1}$.
For $\phi:F\to F$ an endomorphism, a subgroup system $\mathcal{A} = \{A_1,\ldots , A_l\}$ is \emph{$[\phi]$-invariant}, or \emph{invariant} if $\phi$ is implicit, if $\mathcal{A}$ carries $\phi(\mathcal{A}) = \{\phi(A_1),\ldots , \phi(A_l)\}$, and is \emph{properly $[\phi]$-invariant}, or \emph{properly invariant}, if the inclusions $\phi(A_i) \leq xA_jx^{-1}$ are all proper.

Any $F$-labelled graph $\Gamma$ \emph{represents} a subgroup system, as it corresponds to a map of graphs $f:\Gamma\to R_F$ where $R_F$ is the rose with vertex $v$, and so the set $f^{-1}(v)=\{v_1, \ldots, v_{l}\}$ of $\Gamma$ gives a subgroup system $\{f_*(\pi_1(\Gamma, v_1)), \ldots, f_*(\pi_1(\Gamma, v_{l}))\}$ of $F=\pi_1(R_F)$.
In particular,
we can talk about one iterate carrying another.
Furthermore, we can view components of iterates as subgroup systems, so can talk about an element $u\in F$ being carried by a component $\mathcal{C}$ of an iterate $\Lambda_j$.


\p{Stabilisation}
%
Suppose that $\hat\Lambda_j$ is non-empty for all $j\geq1$, and write $k_0:= 2(\rk(F)-1)^2$.
Mutanguha proved that for all $k\geq k_0$, the graph $\hat\Lambda_k$ consists of cyclic components \cite[Lemma 5.1.4]{Mutanguha2021dynamics}.
He then proved that there exists some $k\geq k_0$ such that a component of $\hat{\Lambda}_k$ represents a properly $[\phi]$-invariant proper cyclic subgroup system \cite[Proposition 5.1.5]{Mutanguha2021dynamics}.
We need the following (algorithmic and structural) strengthening of this result.
Mutanguha explained the proof of this result to me in a private communication; the proof is reproduced below with his kind permission.

An iterate $\hat{\Lambda}_k$ is \emph{stable} if every component of $\hat{\Lambda}_k$ represents a properly $[\phi]$-invariant proper cyclic subgroup system.
%
%
%
%
%
Let $|\phi|$ be the word-norm of $\phi$ (i.e. max. word-length of $\phi(a)$ over the basis elements used to define word-length).

\begin{proposition}[Mutanguha]
\label{prop:JPMremark}
There exists an algorithm with input an automorphic expansion $\widetilde{\phi}=(f,\Psi)$ defined on a graph $(\Delta, \mathcal{D})$,
and output a number $k = k(f,\Psi)$
such that $\hat \Lambda_k$ is stable or empty, and a properly invariant cyclic subgroup system (for the induced injective endomorphism of $F$) for each component of $\hat \Lambda_k$.
\end{proposition}

\begin{proof}
Let
$N = N(\Delta)$
be the number of edges in $\Delta$.
Let $k_0 = 2(\rk(F)-1)^2$ and
$M = M(\Delta)$
be the number of vertices in $\Delta$.
Let $k_1 = \max(k_0, 1+M)$ and $N_1 = N_1(f, \Psi)$ be the number of components in $\hat \Lambda_{k_1}$.
If $N_1 = 0$, then set $k = k_1$ and the result holds vacuously (output the empty cyclic subgroup system).

Note that for any edge $e$ in $\Delta$, the edge-path $f^N(e)$ has at least $2$ edges as $\widetilde{\phi}=(f, \Psi)$ is an expansion.
Therefore, for all $i\in\mathbb{N}$, if $n > N*\log_2(i)$, then $f^n(e)$ has at least $i$ edges.

Assume $N_1 > 0$.
Any component $C_i \in \hat \Lambda_{k_1}$ is cyclic (as $k_1 \geq k_0$ \cite[Lemma 5.1.4]{Mutanguha2021dynamics}) and hence orientable.
Fix an orientation on the components of $\hat \Lambda_{k_1}$.
From the definition of $\hat\Lambda_{k_1}$ in terms of pullback maps, an oriented component $C_i \in \hat \Lambda_{k_1}$ is an ordered pair $(c_i, c_i')$ of oriented immersed loops in $(\Delta, \mathcal{D})$.
For $n >  0$, any component $C$ of $\hat \Lambda_{k_1+n}$ is again cyclic and an oriented $C$ is a pair $(c, c')$ of oriented immersed loops in $(\Delta, \mathcal{D})$ such that
there exists a component $C_i=(c_i, c_i')$ of $\hat \Lambda_{k_1}$ and exponents $d = d(C)$, $d' = d'(C)$ such that
$\widetilde{\phi}^n(c) \sim c_i^d$ and $\widetilde{\phi}^n(c') \sim (c_{i}')^{d'}$
i.e. the component $C$ of $\hat \Lambda_{k_1+n}$ covers a component of $\hat \Lambda_{k_1}$.
As $c, c'$ are immersed loops in $(\Delta, \mathcal{D})$ and as $\widetilde{\phi}$ is an immersion, $c_i^d$ and $(c_i')^{d'}$ are also immersed loops in $(\Delta, \mathcal{D})$ and so $\widetilde{\phi}^n(c)$ is a cyclic permutation of $c_i^d$, and $\widetilde{\phi}^n(c')$ is a cyclic permutation of $(c_i')^{d'}$.
By reversing the orientation on $C$ if necessary, assume $d > 0$.
The loops $c$ and $c_i$ in $(\Delta, \mathcal{D})$ are not in a vertex rose (as $k_1 > M$, this follows from the third paragraph in the proof of \cite[Proposition 5.1.5]{Mutanguha2021dynamics}), i.e. each has a $\Delta$-edge.
Let $r_i$ be the root of $c_i$, $i = 1,\ldots,N_1$ (i.e. $c_i=r_i^{p_i}$, $p_i>0$ maximal).
Let $L_i = L(r_i)$ be the length of $r_i$ (in $\Delta$) and $L = \max(L_1, \ldots, L_{N_1})$.
By the above, for $n > 0$ and any component $C = (c, c')$ of $\hat \Lambda_{(k_1+n)}$, $\widetilde{\phi}^n(c)$ is some cyclic permutation of a positive power of $r_i$, where $i = i(C)$ and $C$ is oriented appropriately.
Moreover, if $q > 0$, $t > 0$, $n - t > N*\log_2(L(q+1))$, and $e$ is an oriented $\Delta$-edge in $\widetilde{\phi}^t(c)$, then the oriented edge-path $\widetilde{\phi}^{n-t}(e)\subseteq \widetilde{\phi}^n(c)\subseteq(\Delta, \mathcal{D})$ contains $r_i^q$ as a subpath.
Let $S_i = S(r_i)$ be the number of (nontrivial) closed subpaths in all cyclic permutations of $r_i$ and $S = 1 + \max(S_1, \ldots, S_{N_1})$.
Let $k = k_1 + (1+2N) + N*\lceil\log_2(L(S+1))\rceil$; this is the computable number $k = k(f, \Psi)$.

Let $C = (c, c')$ be a given component of $\hat \Lambda_k$.
Let $i = i(C) \leq N_1$, $d = d(C) > 0$ (pick right orientation on $C$ for this), and $n = k - k_1$. 
By the pigeonhole principle, at least two of the oriented immersed loops $c, \widetilde{\phi}(c), \ldots, \widetilde{\phi}^{2N}(c)$ in $(\Delta, \mathcal{D})$ share some oriented $\Delta$-edge $e$ (recall that $c$ contains a $\Delta$-edge as $k > M$ and so do all its $\widetilde{\phi}$-iterates as $\widetilde{\phi}=(f, \Psi)$ is an expansion).
Assume the oriented $\Delta$-edge $e$ is contained in $\widetilde{\phi}^{t_1}(c)$ and $\widetilde{\phi}^{t_2}(c)$ for some integers $0 \leq t_1 < t_2 \leq 2N$. 
Let $0 < t = t_2 - t_1 \leq 2N$.
Since $n - t_2 > N*\log_2(L(S+1))$, the oriented edge-path $\widetilde{\phi}^{n-t_2}(e)$ contains $r_i^S$.
There is a segment $p$ in $e$ such that $\widetilde{\phi}^{n-t_2}(p) = r_i^S$.
The oriented edge-path $\widetilde{\phi}^t(r_i^S) =  \widetilde{\phi}^t(\widetilde{\phi}^{n-t_2}(p)) = \widetilde{\phi}^{n-t_1}(p)$  is a closed subpath of $\widetilde{\phi}^{n-t_1}(\widetilde{\phi}^{t_1}(c)) = \widetilde{\phi}^n(c) \sim c_i^d$.
It follows that, for $1 \leq j \leq S$, $\widetilde{\phi}^t(r_i^j)$ is a closed subpath of $\widetilde{\phi}^n(c)$ and $\widetilde{\phi}^n(c)$ is a positive power of some cyclic permutation $\overline{r}_i$ of $r_i$.
There are nontrivial closed subpaths $e_j$  of $\overline{r}_i$ and positive integers $d_j$ such that $\widetilde{\phi}^t(r_i^j)e_j = \overline{r}_i^{d_j}$ and $d_1 < d_2 < \ldots < d_S$.
By definition of $S$ and the pigeonhole principle, there is some $j < j' \leq S$ for which $e_j = e_j' = E$. Therefore, $\widetilde{\phi}^t(r_i^j)E =\overline{r}_i^{d_j}$, $\widetilde{\phi}^t(r_i^{j'})E = \overline{r}_i^{d_{j'}}$, and $\widetilde{\phi}^t(r_i)^{j'-j} = \widetilde{\phi}^t(r_i^{j'-j}) = \overline{r}_i^{d_{j'}-d_j} \sim r_i^{d_{j'}-d_j}$. 
By minimality of roots, $\widetilde{\phi}^t(r_i) \sim r_i^{s}$ for some $s>1$; the power $s$ is proper since $\widetilde{\phi}=(f, \Psi)$ is an expansion and $r_i$ is not in a vertex rose.
So $r_i$ determines a conjugacy class of a properly invariant cyclic subgroup system in $F$ and this concludes the proof.
\end{proof}

This result has the following algebraic interpretation, which is what we require in the sequel.

\begin{lemma}
\label{prop:JPM5.5}
Suppose $\phi$ is an injective endomorphism of $F$. There is a computable positive integer $k = k(\phi)$ such that if $c \in F$ is carried by a component of $\hat \Lambda_k$, then $\phi^{t_c}(c) \sim c^{d_c}$ for some integers $1 \leq t_c \leq 6 \rk(F) - 6$ and $2 \leq d_c \leq |\phi|^{6 \rk(F) - 6}$.
\end{lemma}

\begin{proof}
Compute an automorphic expansion $(f, \Psi)$ defined on the graph $(\Delta, \mathcal{D})$ representing $\phi$.
Let $k$ be the integer output by Proposition \ref{prop:JPMremark} and let $k_1$ be as in the proof of Proposition \ref{prop:JPMremark}, so $k_1 = \max(k_0, 1+M)$ where $k_0 = 2(\rk(F)-1)^2$ and $M$ is the number of vertices in $\Delta$.
The proof shows that $k_1\leq k$,
so letting $n = k - k_1$ we have $0 \leq n < k$.

Let $c$ be an element in $F$ carried by a component of $\hat \Lambda_k$.
So there are elements $x$ and $x'$ in $F$ such that $c \sim \phi^k(x) \sim \phi^k(x')$ but $\phi^{k-1}(x)$ and $\phi^{k-1}(x')$ are not conjugate.
Let $r$ be the root of the cyclic reduction of $\phi^n(x)$, and let $N$ be the number of edges in $\Delta$ (as in the proof of Proposition \ref{prop:JPMremark}).
Note that $N$ is less than or equal to the number of edges in the blow-up of the graph of roses $(\Delta, \mathcal{D})$, and so $N\leq 3 \rk(F) - 3$.
Then $\phi^t(r) \sim r^d$ for some $1 \leq t \leq 2N\leq6 \rk(F) - 6$ and $d \geq 2$, by the proof of Proposition \ref{prop:JPMremark}.
So $\phi^t(c) \leq c^d$ since $c \sim \phi^k(x) = \phi^{k_1}(\phi^n(x)) \sim \phi^{k_1}(r)$.
Then $|\phi^t(r)| \leq |r| * |\phi|^t$. Yet $|\phi^t(r)| = |r^d| = |r| * d$ as $r$ is cyclically reduced. 
So $d \leq |\phi|^t \leq |\phi|^{6 \rk(F) - 6}$ as needed.
\end{proof}

\subsection{Elements carried by $\hat \Lambda_k$}
We now give an algorithm to determine if some endomorphic image $\phi^i(w)$ of an element $w\in F$ is carried by the stable iterate $\hat \Lambda_k$.
An element of $F$ is \emph{primitive} if it is contained in a free basis of $F$.
We begin with a strong result which says that if an element is mapped to a proper power of itself by endomorphisms of $F$, up to conjugacy, then it is primitive in $F$.
The stronger result regarding subgroups, that if an element $x\in H\leq F$ is a proper power in $F$ but not in $H$ then $x$ is primitive in $H$, is true for subgroups $H$ of rank two (due independently to Baumslag and Steinberg \cite{Baumslag1965Residual, Steinberg1971equations}) but false in general \cite{Rosenberger1988minimal}.
As the Baumslag--Solitar subgroups $\operatorname{BS}(1, d)$, $d\geq1$, correspond precisely to identities of the form $\phi^t(c)\sim c^d$ \cite[Theorem 5.3.5]{Mutanguha2021dynamics}, this lemma can be thought of as classifying how such subgroups algebraically occur in ascending HNN-extensions of free groups.

A subgroup $H$ of a group $G$ is \emph{malnormal} if for $g\in G$, $g^{-1}H g\cap H\neq1$ implies that $g\in H$.
It is a standard fact that maximal cyclic subgroups of free groups are malnormal.

\begin{lemma}
\label{lem:primitivity}
Suppose $\phi$ is an injective endomorphism of $F$.
If $\phi^t(c)\sim c^d$ with $d>1$ and $c\in F$ non-trivial, then $c$ is a power of a primitive element of $F$.
\end{lemma}

\begin{proof}
As we are working up to conjugacy, we may assume that $c$ is cyclically reduced.
Decompose $c$ as $c_0^{p}$ for $p>0$ maximal, which gives $\phi^t(c_0^{p})\sim c_0^{p d}$.
As roots are unique in free groups, we therefore have $\phi^t(c_0)\sim c_0^{d}$, $d>1$.
Then $\langle c_0\rangle$ is maximal cyclic in $F$, and hence malnormal in $F$.
Note that as $c$ is cyclically reduced, $c_0$ is cyclically reduced.
We may replace $\phi$ with $\psi:=\phi^t\chi$ for some inner automorphism $\chi$ such that $\psi(c_0)= c_0^d$.

For $\psi: F\to F$ our injective endomorphism, consider the graph of roses $(\Delta, \mathcal{D})$ and associated homotopy equivalence $\alpha$ and automorphic expansion $\widetilde{\psi}=(f,\Psi)$ given by Theorem \ref{thm:JPM}.
%
As $\widetilde{\psi}$ is an expansion, there exists a number $m\gg1$ such that in $\widetilde{\psi}^{m}(\Delta, \mathcal{D})$, every edge of $f(\Delta)=\Delta$ which is traversed by some immersed loop $\gamma$
has a label which contains $c_0^2$ (up to inversion).
Let $\ell>0$ be minimal such that $c_0^{\ell}\in\im(\psi^m)$, and write $\gamma$ for the immersed (unbased) loop in $\widetilde{\psi}^{m}(\Delta, \mathcal{D})$ which corresponds to $[c_0^{\ell}]$.

The loop $\gamma$ does not contain any edge $e$ of $\Delta$ twice.
We prove this algebraically, using the fact that, as noted above, $\langle c_0\rangle$ is malnormal in $F$.
Decompose $\gamma$ as $\gamma_0e\gamma_1e^{\epsilon}\gamma_2$, $\epsilon=\pm1$, which gives a labeling of $\gamma$ of the form $[g_0c_0^2g_1c_0^{\epsilon 2}g_2]$.
Then $[c_0^\ell]=[g_0c_0^2g_1c_0^{\epsilon 2}g_2] = [c_0^{\epsilon 2}g_2g_0c_0^2g_1]$, and so $c_0^{\epsilon 2}g_2g_0c_0^2g_1$ must be a proper power of $c_0^{\epsilon}$.
As $\langle c_0\rangle$ is malnormal and abelian, it follows that $\epsilon=1$.
The two occurrences of $e$ have labels $h_0c_0^ih_1$ and $k_0c_0^ik_1$, respectively, and malnormality of $\langle c_0\rangle$ gives us that $k_0=h_0$ and $k_1=h_1$.
It follows that we can shortcut the $\gamma_1$ part of the loop: $\gamma_0e\gamma_2$ is an immersed loop in $\widetilde{\psi}^{m}(\Delta, \mathcal{D})$, which corresponds to $[c_0^{\ell'}]$ with $0<{\ell'}<\ell$.
This contradicts the minimality of $\ell$.
We therefore conclude that $\gamma$ does not contain any edge $e$ of $\Delta$ twice, as claimed.

Now consider the blow-up $\top(\Delta, \mathcal{D})$ of $(\Delta, \mathcal{D})$, with the labeling obtained from $\widetilde{\psi}^{m}(\Delta, \mathcal{D})$.
By the above paragraph, there exists a simple closed curve $\delta\subset \top(\Delta, \mathcal{D})$ and an immersed loop $\gamma'\subset \top(\Delta, \mathcal{D})$ with label $c_0^{\ell}$ such that $\delta\subset\gamma'$ and $\gamma'$ contains each edge of $\delta$ precisely once; here $\gamma'=\top(\gamma)$ and $\delta$ corresponds to the path in $\Delta$ traversed by $\gamma$.
As $c_0^{\ell}\in\im(\psi)$, there exists a vertex $u$ of $\top(\Delta, \mathcal{D})$ such that $\pi_1(\top(\Delta, \mathcal{D}), u)=\im(\psi^m)$ and such that $\gamma'$ based at $u$ is precisely $c_0^{\ell}$ under this equality.
Consider a spanning tree $T$ of $\top(\Delta, \mathcal{D})$; this gives a basis $\mathcal{B}$ for $\pi_1(\top(\Delta, \mathcal{D}), u)$ consisting of simple closed curves, each containing a unique edge of $\top(\Delta, \mathcal{D})\setminus T$.
As $\delta$ is a simple closed curve, there exists some edge $e$ from $\delta$ which is not contained in $T$; let $\delta'\in \mathcal{B}$ be the unique element of this basis which contains the edge $e$.
As $\gamma'$ contains the edge $e$ precisely once, its decomposition as a product of elements from $\mathcal{B}$ contains the element $\delta'$ precisely once; this means that we may replace $\delta'$ in the basis $\mathcal{B}$ with $\gamma'$ to get a new basis.
It follows that $\gamma'$ is primitive in $\pi_1(\top(\Delta, \mathcal{D}), u)$, and hence $c_0^{\ell}$ is primitive in $\im(\psi^m)$.
As $\psi(c_0)$ is equal to some power of $c_0$,
we have that $\psi^{m}(c_0)$ equals some power of $c_0$, which in turn gives us that $\psi^{m}(c_0^\ell)$ is some power of a primitive in $\im(\psi^{m})$.
As $\psi$ is injective, some power of $c_0$ is a power of a primitive in $F$.
As $c_0$ has no proper roots, it must itself be primitive in $F$.
The result follows.
\end{proof}

The following proof uses
the free factor systems
of Bestvina--Feighn--Handel
\cite[Section 2.6]{Bestvina2000Tits}:
A \emph{nontrivial free factor system} of $F$ is a subgroup system $\mathcal{A} = \{A_1, \ldots, A_l\}$ of $F$ such that each $A_i$ is a free factor of $F$, $A_i \cap A_j$ is trivial if $i \neq j$, and $\langle A_1, \ldots , A_l\rangle$ is a free factor of $F$.
Note that $\langle A_1, \ldots , A_l\rangle\cong A_1\ast\cdots \ast A_l$.

\begin{lemma}
\label{lem:p-equals-0-subgroup-membership}
There exists an algorithm which determines for a given element $w\in F$ and a given monomorphism $\phi:F\to F$, if there exists some $i\in\mathbb{N}$ such that $\phi^i(w)$ is carried by the stable iterate $\hat \Lambda_k$.
\end{lemma}

\begin{proof}
The stable iterate $\hat \Lambda_k$ has only finitely many components, and these are all loops.
As we can compute $\hat \Lambda_k$, for each component $\mathcal{C}$ of $\hat \Lambda_k$ we can find a representative $c\in F$, so closed loops in $\mathcal{C}$ correspond to powers of $c$.
It is sufficient to prove the result for a fixed component $\mathcal{C}$, with representative $c$, i.e. provide an algorithm to determine if there exists some $i\in\mathbb{N}$, $j\in\mathbb{Z}$ such that $\phi^i(w)\sim c^j$.

By Lemma \ref{prop:JPM5.5}, there are computable numbers $t$ and $d$ such that $\phi^t(c)\sim c^d$, so then by Lemma \ref{lem:primitivity} we have that $c=c_0^{\ell}$ for some primitive element $a$ of $F$ and some $\ell>0$, so $A:=\langle a\rangle$ is a free factor of $F$.
The identity $\phi^t(c)\sim c^d$ also implies that $\phi^{t}(a)^{\ell}\sim a^{\ell d}$, so by applying the fact that free groups have unique roots we have $\phi^{t}(a)\sim a^{d}$.
Therefore, $\phi^t(A)$ is carried by $A$.
Then its preimage $\phi^{-t}(A)$ is a factor system $\mathcal{A}_1$ (since $\phi^{t}$ is injective), consisting of $A$ and a bunch of cyclic free factors whose images under $\phi^t$ are carried by $A$ \cite[p10]{Mutanguha2021dynamics}.
Note that if $\phi^k(x)$ is carried by the free factor $F_{\mathcal{A}}$ corresponding to $\mathcal{A}$, then $x$ is carried by the free factor $F_{\mathcal{A}_1}$ corresponding to $\mathcal{A}_1$.
Keeping taking preimages gives an ascending chain of free factor systems $\mathcal{A}_1\geqslant \mathcal{A}_2\geqslant\cdots$, where elements $A_j$ of $\mathcal{A}_{i}$ are also components of $\mathcal{A}_{i+1}$, and if $\phi^{k}(x)$ is carried by the free factor $F_{\mathcal{A}_i}$ corresponding to $\mathcal{A}_i$, then $x$ is carried by the free factor $F_{\mathcal{A}_{i+1}}$ corresponding to $\mathcal{A}_{i+1}$
As each $\mathcal{A}_i$ is a cyclic free factor system, it can contain at most $\rk(F)$ components, so this chain must eventually stabilise.
If $\mathcal{A}_i=\mathcal{A}_{i+1}$ then this chain has stabilised; hence it must stabilise in at most $\rk(F)$ steps.
This means that we can compute the free factor system $\overline{\mathcal{A}}=\{A_1, \ldots, A_m\}$ at which this chain stabilises.
Write $F_A$ for the associated free factor, and we have that if $\phi^{kj}(x)$ is carried by $F_{\mathcal{A}}$ then $x$ is carried by $F_A$.
This last condition means that if $\phi^i(w)$ is carried by the component $\mathcal{C}$ of $\hat \Lambda_k$, then one of $w, \phi(w), \ldots, \phi^{k-1}(w)$ is carried by the free factor $F_A$.

The algorithm required for the theorem is therefore to loop over the components of $\hat \Lambda_k$, and for each component $\mathcal{C}$ compute the free factor system $\overline{\mathcal{A}}$.
This gives us a basis for $F_A$, and we can use this basis to determine if any of $w, \phi(w), \ldots, \phi^{k-1}(w)$ are carried by $F_A$ (for example, via Stallings' graphs), and by the above this is equivalent to the property we are seeking to determine.
\end{proof}


\section{Generalising Brinkmann's conjugacy algorithm}
\label{sec:prelim}
We now prove the generalisation of Brinkmann's conjugacy algorithm required for Theorem \ref{thm:UltraTwCP}.
Brinkmann's algorithm outputs some $p\in\mathbb{Z}$ such that $\phi^p(u)\sim v$, but if $\phi$ is non-surjective then $\phi^{p}$ is only an endomorphism of $F$ if $p\geq0$.
We in fact prove two generalisations of this algorithm, which are two different ways of overcoming the $p<0$ issue.
Firstly, we consider one exponent $p\in\mathbb{N}$ and the identity $\phi^p(u)\sim v$, as Brinkmann did.
Secondly, we consider two exponents $p,q\in\mathbb{N}$ and the identity $\phi^p(u)\sim\phi^q(v)$.
We apply the first generalisation to prove the second, main generalisation.

\subsection{Direct generalisation of Brinkmann's conjugacy algorithm}
We now replace the ``$p\in\mathbb{Z}$'' in Brinkmann's conjugacy algorithm with $p\in\mathbb{N}$; this is a direct generalisation of Brinkmann's algorithm.

\begin{lemma}
\label{lem:p-equals-0}
There exists an algorithm with input an injective endomorphism $\phi$ of a free group $F$ and two elements $u, v\in F$, and output either an exponent $p\in \mathbb{N}_0$ such that $\phi^p(u)\sim v$,
or $\mathrm{no}$ if no such exponent exists.
\end{lemma}

\begin{proof}
Let $\phi: F\to F$ be an injective endomorphism, and consider the graph of roses $(\Delta, \mathcal{D})$ and associated homotopy equivalence $\alpha$ and automorphic expansion $\widetilde{\phi}=(f,\Psi)$ given algorithmically by Theorem \ref{thm:JPM}.
We can compute which of the finitely many vertex roses of $(\Delta, \mathcal{D})$ are $\Psi$-periodic, as for example the period is bounded by $d$, the number of roses in the finite graph $\mathcal{D}$.

For $w\in F$, we view the conjugacy class $[w]$ as an (unbased) loop $\gamma_w$ in $(\Delta, \mathcal{D})$.
As $\widetilde{\phi}$ is an expansion, there exists a computable number $m_v$ such that if $\gamma_v$ is a loop in $\widetilde{\phi}^{j}(\Delta, \mathcal{D})$ for $j>m_v$ then $\gamma_v$ is wholly contained a vertex rose $R$ of $\mathcal{D}$.
Define $m:=\max(m_v, k, d)$, where $k$ is the constant from Lemma \ref{prop:JPM5.5} and $d$ is the number of roses in the finite graph $\mathcal{D}$.
As $m$ is computable, we are concerned with determining if there exists some $p\geq m$ such that $\phi^p(u)\sim v$.
If $\gamma_v$ is not in any $\Psi$-periodic vertex rose, then there is no $p\geq m$ such that $\gamma_v$ is a translate of $\gamma_{\phi^p(u)}$, and hence no $p\geq m$ such that $\phi^p(u)\sim v$.
Suppose therefore that $\gamma_v$ is in a $\Psi$-periodic vertex rose $R$.
This assumption implies that $\gamma_v$ is contained in $\im(\widetilde{\phi}^j)$ for all $j\geq0$, and so if there exists some $p\geq m$ such that $\gamma_v$ is a translate of $\gamma_{\phi^p(u)}$, then there exists some $\gamma_{v_0}$ contained in some periodic vertex rose $R'$ such that $\widetilde{\phi}^p(\gamma_{v_0})=\gamma_{\phi^p(v_0)}$ is a translate of $\gamma_{\phi^p(u)}$.
Assuming that $\gamma_{v_0}$ is not a translate of $\gamma_u$ and that $p\geq m$ is minimal with this condition, we see that $\phi^p(v_0)$ is carried by $\hat{\Lambda}_p$, and as $p\geq m\geq k$ it is therefore carried by $\hat{\Lambda}_k$.
Lemma \ref{prop:JPM5.5} gives us that $\phi^t(v_0)\sim v_0^d$ for some $d\geq2$, but as $\gamma_{v_0}$ is in the $\Psi$-periodic vertex rose $R'$, where $\widetilde{\phi}$ acts homotopically, this is not possible.
This contradiction means that there exists some $j<m$ such that $\gamma_{\phi^j(v_0)}$ is a translate of $\gamma_{\phi^j(v_0)}$, and for all $p\geq m$ the loop $\gamma_{\phi^p(u)}$ is contained in a $\Psi$-periodic vertex rose, up to translation.
Suppose therefore that $\gamma_u$ eventually lands in the same $\Psi$-periodic vertex rose as $\gamma_v$, and write $i \geq 0$ for the smallest integer such that $v$ and $\phi^i(u)$ are represented by unbased loops in the same $\Psi$-periodic vertex rose $R$.
Note that there are finitely many such $R$, and that we can compute them all.

Fix one such $R$, and let $\ell \geq1$ be its $\Psi$-period.
As $\widetilde{\phi}$ is automorphic, i.e. its restriction to the periodic vertex roses is a homotopy equivalence, the problem of determining if there exists some $p\geq m$ such that $\phi^p(u)\sim v$ reduces to determining if there exists any $p'\in\mathbb{N}_0$ such that the loop $\Psi^{\ell p'+i}(\gamma_u)$ is homotopic to $\gamma_v$.
Brinkmann's conjugacy algorithm determines if there exists an exponent $p_0\in\mathbb{N}_0$ such that either
$\Psi^{\ell p_0+i}(\gamma_u)$ is homotopic to $\gamma_v$
or
$\Psi^i(\gamma_u)$ is homotopic to $\Psi^{\ell p_0}(\gamma_v)$
\cite[Theorem 6.2]{Brinkmann2010Detecting}.
Suppose Brinkmann's conjugacy algorithm outputs some $p_0\in\mathbb{N}_0$ such that $\Psi^i(\gamma_u)$ is homotopic to $\Psi^{\ell p_0}(\gamma_v)$.
There additionally exists some $p_1\in\mathbb{N}_0$ with $\Psi^{\ell p_1+i}(\gamma_u)$ homotopic to $\gamma_v$ if and only if there exists some $p_2\in\mathbb{Z}\setminus\{0\}$ such that $\Psi^{\ell p_2}(\gamma_{v})$ is homotopic to $\gamma_{v}$. For one direction take $p_2:=p_0+p_1$, while for the other take $p_1:=j|p_2|-p_0$ for any large enough $j$ (e.g. $j:=2p_0$ works).

Therefore, our algorithm is as follows:
Run the algorithm of Mutanguha (Theorem \ref{thm:JPM}), and compute the number $m$.
Determine for all $p\leq m$ if $\phi^p(u)\sim v$ in $F$, and if so output the relevant exponent $p$.
If no such exponent exists, then
first find all the $\Psi$-periodic vertex roses of $\mathcal{D}$.
Loop across these roses $R$.
For each such $R$, find its $\Psi$-period $\ell$ and determine if it contains $\gamma_u$, and find all $i\leq \ell$ such that $\gamma_v$ is in $\Psi^i(R)$.
Loop across these indices $i$.
For each such $i$, determine via Brinkmann's conjugacy algorithm if there exists an exponent $p_0\in \mathbb{N}_0$ such that either
$\Psi^{\ell p_0+i}(\gamma_u)$ is homotopic to $\gamma_v$
or
$\Psi^i(\gamma_u)$ is homotopic to $\Psi^{\ell p_0}(\gamma_v)$.
If $\Psi^{\ell p_0+i}(\gamma_u)$ is homotopic to $\gamma_v$ then we have $\phi^{\ell p_0+i}(u)\sim v$, so output $p:=p_0$.
If $\Psi^i(\gamma_u)$ is homotopic to $\Psi^{\ell p_0}(\gamma_v)$ then determine, via an algorithm of Feighn and Handel \cite[Remark 9.5, Lemma 9.6(2)]{Feighn2018algorithmic}, if there exists an exponent $p_2\in\mathbb{Z}\setminus\{0\}$ such that $\Psi^{\ell p_2}(\gamma_{v})$ is homotopic to $\gamma_{v}$, and if so output $p:=\ell p_0(2|p_2|-1)+i$ and terminate the algorithm.
If both loop ends without outputting any exponent $p$, then output ``$\mathrm{no}$'' as there is no exponent $p\in \mathbb{N}_0$ such that $\phi^p(u)\sim v$ in $F$.
\end{proof}

\subsection{Two-sided generalisation of Brinkmann's conjugacy algorithm}
\label{sec:Brinkmann}
We now build on Lemma \ref{lem:p-equals-0} to give a second generalisation of Brinkmann's conjugacy algorithm,
which is the version appearing in Theorem \ref{thm:UltraTwCP}.
This versions solves equations of the form $\phi^p(u)\sim\phi^q(v)$ for $p, q\geq0$.

\p{Conjugator not in \boldmath{$\im(\phi)$}}
We start with an algorithm to solve $x^{-1}\phi^p(u)x=\phi^q(v)$ for a conjugator $x\not\in\im(\phi)$, and to do this we apply Section \ref{sec:Stabilisation}.
Firstly, we give the algorithm for $p, q$ larger than a computable constant.

For a number $i\in\mathbb{N}_0$, we write $\mathbb{N}_{< i}$ for the set of natural numbers strictly less than $i$, and $\mathbb{N}_{\geq i}$ for the set of natural numbers greater than or equal to $i$.

\begin{lemma}
\label{lem:large-p-q}
There exists an algorithm with input an injective, non-surjective endomorphism $\phi$ of a free group $F$ and two elements
$u, v\in F$,
and output a number $d\in\mathbb{N}_0$ and either a pair $(p, q)\in \mathbb{N}_{\geq d}\times\mathbb{N}_{\geq d}$ with $p\geq q\geq d$ such that
there exists some element $x\in F\setminus\im\phi$ satisfying $x^{-1}\phi^p(u)x=\phi^q(v)$
or $\mathrm{no}$ if no such pair exists.
\end{lemma}

\begin{proof}
Start the algorithm by computing the integer $k = k(\phi)$ and the iterated pullback $\hat \Lambda_k$, given to us by Lemma \ref{prop:JPM5.5}.
Recall that $\hat\Lambda_k$ carries $\hat\Lambda_l$ for all $l\geq k$, and that $\hat\Lambda_k$ consists of disjoint loops.
Next, we use Lemma \ref{lem:p-equals-0-subgroup-membership} to determine if there exists any $i, j$ such that $\phi^i(u), \phi^j(v)$ are carried by $\hat\Lambda_k$, which we may do as $\hat\Lambda_k$ consists of disjoint loops.
As $\hat\Lambda_k$ carries $\hat\Lambda_l$ for all $l\geq k$, if $\phi^i(u)$ or $\phi^j(v)$ is not carried by $\hat\Lambda_k$ then it is not carried by $\hat\Lambda_l$ for any $l\geq k$.

Suppose no $\phi^i(u)$, $i\geq0$, is carried by any component of $\hat \Lambda_k$.
If there exists $p\geq q$ and some $x\in F\setminus\im\phi$ satisfying $x^{-1}\phi^p(u)x=\phi^q(v)$, then as $p\geq q$ we have $\phi^{p}(u)\in\im\phi^q$, and so $\phi^{p}(u)$ is carried by $\hat \Lambda_q$.
Hence $q< k$, so output $d:=k$ and $\mathrm{no}$.

Suppose no $\phi^j(v)$, $j\geq0$, is carried by any component of $\hat \Lambda_k$.
If there exists $p\geq q$ and some $x\in F\setminus\im\phi$ satisfying $x^{-1}\phi^p(u)x=\phi^q(v)$, then as $p\geq q$ this rewrites to $\phi^{q}(\phi^{p-q}(u))\sim\phi^{q}(v)$, and so $\phi^{q}(v)$ is carried by $\hat \Lambda_q$.
Hence $q< k$, so output $d:=k$ and $\mathrm{no}$.

Suppose that $i, j\geq0$ are such that $\phi^i(u)$ and $\phi^j(v)$ are each carried by some component of $\hat \Lambda_k$.
Compute the constants $d_u, d_v, t_u, t_v$ given to us by Lemma \ref{prop:JPM5.5}, so $\phi^{t_u+i}(u)\sim\phi^i(u)^{d_u}$ and $\phi^{t_v+j}(v)\sim\phi^j(v)^{d_v}$.
We set $d:=\max(d_ud_v, i, j)$, which is therefore computable.
We can then rewrite $\phi^p(u)$ as follows, where $i\leq r_p<d$ and $n_u=d_u^{d_v}$:
\begin{align*}
\phi^p(u)
&=\phi^{ad+r_p}(u)\\
&=\phi^{r_p}(\phi^{ad}(u))\\
&=\phi^{r_p}(u)^{n_u^a}
\end{align*}
As there are only finitely many choice for the remainder $r_p$, we see that as $p$ varies, $\phi^p(u)=u_l^{n_u^a}$ where $u_l:=\phi^l(u)$ is one of finitely many, computable elements $u_1, \cdots, u_d\in F$.
Similarly, we can decompose $q$ as $jd+r_q$, where $0\leq r_q<d$, and we have $\phi^q(v)=\phi^{r_q}(v)^{n_v^b}$ and so $\phi^q(v)=v_m^{n_v^b}$ where $v_m:=\phi^m(v)$ is one of finitely many, computable elements $v_1, \cdots, v_d\in F$.
Therefore, if there exists a pair $(p, q)\in \mathbb{N}_{\geq d}\times\mathbb{N}_{\geq d}$ such that $p\geq q\geq d$ and $\phi^p(u)\sim\phi^q(v)$ then there exist indices $l, m$ such that $u_l^{n_u^a}\sim v_m^{n_v^b}$ and $ad+l\geq bd+m$.
Here, the lists $(u_l), (v_m)$ and the exponents $n_u, n_v$ are computed from the pair $u, v\in F$, and so only the exponents $a, b$ can vary.
On the other hand, if there exist indices $l, m$ and exponents $a, b\in\mathbb{Z}$ such that $u_l^{n_u^a}\sim v_m^{n_v^b}$ and $ad+l\geq bd+m$, then
\[
\phi^{ad+s_l}(u)=u_l^{n_u^a}\sim v_m^{n_v^b}=\phi^{bd+s_m}(v).
\]
Therefore, there exists a pair $(p, q)\in \mathbb{N}_{\geq d}\times\mathbb{N}_{\geq d}$ such that $\phi^p(u)\sim\phi^q(v)$ with $p\geq q\geq d$ if and only if there exist indices $l, m$ and exponents $a, b\in\mathbb{Z}\setminus\{0\}$ such that $u_l^{n_u^a}\sim v_m^{n_v^b}$ and $ad+l\geq bd+m$.

Therefore, simplifying our notation, to prove the theorem it is sufficient to give an algorithm which determines for any given pair of words $x, y\in F$ and pair of integers $\alpha, \beta\in\mathbb{Z}\setminus\{0\}$, if there exist $a, b\in\mathbb{Z}$ such that $x^{\alpha^a}\sim y^{\beta^b}$ under the constraint $ad+l_x\geq bd+m_y$, where $l_x$ and $m_y$ depend on $x$ and $y$ respectively.
This algorithm begins by first finding elements $x_0, y_0\in F$ such that $x_0^{\gamma}=x$ and $y=y_0^{\delta}$ with $\gamma, \delta$ maximal, and then determines if $x_0$ and $y_0$ are conjugate; if they are not conjugate then there are no $a, b$ such that $x^{\alpha^a}\sim y^{\beta^b}$. If $x_0\sim y_0$ then $x^{\alpha^a}\sim y^{\beta^b}$ if and only if $\gamma \alpha^a=\delta \beta^b$ as integers, and we can decide this integer problem subject to the given constraint (e.g. by considering prime factorisation). We therefore have our required algorithm, and the result follows.
\end{proof}

We now combine the algorithm of Lemma \ref{lem:large-p-q} with finitely many applications of the algorithm of Lemma \ref{lem:p-equals-0} to obtain the following.

\begin{lemma}
\label{lem:PowSuperTwCPTECHNICAL}
There exists an algorithm with input an injective, non-surjective endomorphism $\phi$ of a free group $F$ and two elements
$u, v\in F$,
and output either a pair $(p, q)\in \mathbb{N}_{0}\times\mathbb{N}_{0}$ with $p\geq q\geq 0$ such that
there exists some element $x\in F\setminus\im\phi$ satisfying $x^{-1}\phi^p(u)x=\phi^q(v)$
or $\mathrm{no}$ if no such pair exists.
\end{lemma}

\begin{proof}
First input the elements $u, v\in F$ into the algorithm of Lemma \ref{lem:large-p-q}.
If this algorithm outputs a pair $(p, q)\in \mathbb{N}_{\geq d}\times\mathbb{N}_{\geq d}$ such that $p\geq q$ and $\phi^p(u)\sim\phi^q(v)$, then this pair satisfies our restrictions so output this pair.
If no such pair exists then consider the other output $d$ from this algorithm; if there exists a pair $(p, q)\in \mathbb{N}_0\times\mathbb{N}_0$ such that $p\geq q$ and $\phi^p(u)\sim\phi^q(v)$ then we must have $q<d$.
Therefore, for all $0\leq i< d$, input the elements $\phi^i(u), \phi^i(v)$ into the algorithm of Lemma \ref{lem:p-equals-0}.
If any of these algorithms output an exponent $p\in\mathbb{N}_0$ such that $\phi^{p+i}(v)=\phi^p(\phi^i(u))\sim\phi^i(v)$ then output the pair $(p+i, i)$.
Otherwise, there is no pair satisfying the restrictions of this result, and so output ``$\mathrm{no}$''.
\end{proof}

\p{Main generalisation of Brinkmann's conjugacy algorithm}
We now have the generalisation of Brinkmann's conjugacy algorithm which we require for Theorem \ref{thm:CP}.

\begin{proposition}
\label{prop:PowSuperTwCP}
There exists an algorithm with input an injective endomorphism $\phi$ of a free group $F$ and two elements $u, v\in F$, and output either a pair $(p, q)\in \mathbb{N}_0\times\mathbb{N}_0$ with
$\phi^p(u)\sim\phi^q(v)$, or $\mathrm{no}$ if no such pair exists.
\end{proposition}

\begin{proof}
The \emph{stable image} of $\phi$ is the subgroup $\phi^{\infty}(F):=\cap_{i=1}^{\infty}\phi^i(F)$.
The restriction of $\phi$ to $\phi^{\infty}(F)$ is an automorphism of $\phi^{\infty}(F)$ \cite{Imrich1989Endomorphisms}.
Our algorithm here begins by applying an algorithm of Mutanguha to compute a basis for $\phi^{\infty}(F)$ \cite{JPM}, and then using this basis to determine if $u$ and $v$ are contained in this subgroup.

Suppose that one of $u$ or $v$ is contained in $\phi^{\infty}(F)$.
As the restriction of $\phi$ to $\phi^{\infty}(F)$ is an automorphism of $\phi^{\infty}(F)$, an identity $\phi^p(u)=g^{-1}\phi^q(v)g$, $p\geq q\geq0$, holds if and only if
$\phi^{p-q}(u)=\phi^{-q}(g^{-1}\phi^q(v)g)$.
As $\phi$ is injective, $\phi^{-q}(g^{-1}\phi^q(v)g)$ is a well-defined element of $F$, and is conjugate to $v$.
Therefore, an identity $\phi^p(u)\sim\phi^q(v)$, $p\geq q\geq0$, holds if and only if an identity $\phi^p(u)\sim v$, $p\geq 0$, holds.
Symmetrically, an identity $\phi^p(u)=g^{-1}\phi^q(v)g$, $q\geq p\geq0$, holds if and only if an identity $u\sim \phi^q(v)$, $q\geq 0$, holds.
The algorithm of Lemma \ref{lem:p-equals-0} solves these latter two problems.
Therefore, output either the pair $(p, 0)$, or $(0, q)$, or $\mathrm{no}$, as appropriate.

Suppose that neither of $u$ or $v$ is contained in $\phi^{\infty}(F)$.
There therefore exist minimum numbers $m, n\in\mathbb{N}$ such that $u\not\in\phi^m(F)$ and $v\not\in\phi^n(F)$, and moreover these numbers are computable.
Define $u':=\phi^{-(m-1)}(u)$ and $v':=\phi^{-(n-1)}(v)$; as $\phi$ is injective these elements are well-defined, and they are computable and are not contained in $\im\phi$.
Now, there exists $p, q\geq0$ with $\phi^p(u)\sim\phi^q(v)$ if and only if either $u'\sim v'$ or there exists $p', q'\geq0$ and $x\not\in\im\phi$ such that $x^{-1}\phi^{p'}(u')x=\phi^{q'}(v')$.
To see this, note first that
if $u'\sim v'$ then we map apply $\phi^{\max(m, n)}$ to both sides and replace $\phi^m(u')$ with $u$ and $\phi^n(v')$ with $v$ to get the required identity, with $(p, q):=(\max(m, n)-m, \max(m, n)-n)$,
while if
$x^{-1}\phi^{p'}(u')x=\phi^{q'}(v')$ for $x\not\in\im\phi$ then we map apply $\phi^{\max(m, n)}$ to both sides and replace $\phi^m(u')$ with $u$ and $\phi^n(v')$ with $v$ to get the required identity, with $(p, q):=(p'+\max(m, n)-m, q'+\max(m, n)-n)$.
On the other hand, suppose the identity $\phi^p(u)\sim\phi^q(v)$ holds, so $y^{-1}\phi^{p}(u)y=\phi^{q}(v)$ for some $y\in F$.
As $\phi$ is injective, if $y\in\im\phi$ we may apply $\phi^{-1}$ to both sides to obtain an identity $y_1^{-1}\phi^{p-1}(u)y_1=\phi^{q-1}(v)$.
We may repeat this, stopping either when the conjugator $x:=y_i\not\in\im\phi$, as required, or after $\min(p+m, q+n)$ steps.
If we get to $\min(p+m, q+n)$ steps, then we either have an identity $x^{-1}\phi^{a}(u')x=v'$ with $v'\not\in\im\phi$ or an identity $x^{-1}u'x=\phi^a(v')$ with $u'\not\in\im\phi$.
If in either case $a=0$ then $u'\sim v'$, as required, while if $a\neq0$ then, as $\phi^a(u')\in\im\phi$ but $v\not\in\im\phi$ in the first case, or as $u\not\in\im\phi$ but $\phi^a(v')\in\im\phi$ in the second, we have that $x\not\in\im\phi$, as required.
We therefore verify if $u'\sim v'$, and if so output $(p, q):=(\max(m, n)-m, \max(m, n)-n)$, which works by the above.
If they are not conjugate, run the algorithm of Lemma \ref{lem:PowSuperTwCPTECHNICAL} on the pairs $(u', v')$ and $(v', u')$, and if either outputs a pair $(p', q')$ then output $(p, q):=(p'+\max(m, n)-m, q'+\max(m, n)-n)$, which works by the above.
Otherwise, there is no pair satisfying the restrictions of this result, and so output ``$\mathrm{no}$''.
\end{proof}

\subsection{Theorem \ref{thm:UltraTwCP} for injections}
We are now able to prove the injective version of Theorem \ref{thm:UltraTwCP}; this is the version we use in the proof of Theorem \ref{thm:CP}.
For ease of application, we omit the constraint that $p\geq q$.

\begin{proposition}
\label{prop:UltraTwCPINJECTIVE}
There exists an algorithm with input an injective endomorphism $\phi$ of a free group $F$, a natural number $n\in\mathbb{N}_0$ and two elements $u, v\in F$, and output either a pair $(p, q)\in \mathbb{N}_0\times\mathbb{N}_0$ such that $\phi^p(u)$ is $\phi^n$-twisted conjugate to $\phi^q(v)$, or $\mathrm{no}$ if no such pair exists.
\end{proposition}

\begin{proof}
If $n=0$ then the result holds by Proposition \ref{prop:PowSuperTwCP}; we run the algorithm of this result twice, inputting first the ordered pair $(u, v)$, and then $(v, u)$.
If $n>0$ then the result holds by Proposition \ref{prop:ExtendedTwCP}.
\end{proof}

A reader interested only in the Conjugacy Problem may now skip straight to Section \ref{sec:MainProofs} and the proof of Theorem \ref{thm:CP}.


\section{Dynamics of endomorphisms}
\label{sec:MainDynamicsProofs}
We now generalise the results of the previous sections on dynamics of monomorphisms to all endomorphisms.
We also prove versions where conjugacy is replaced by equality.

\subsection{From non-injective to injective}
We start with a technical lemma which allows us to jump from a non-injective endomorphism of $F$ to an injective endomorphism of a free factor $F'$ of $F$, and is based on a cute result of Imrich and Turner \cite[Theorem 2]{Imrich1989Endomorphisms}.
A \emph{retaction map} is a map from a group $G$ to a subgroup $H<G$, so $\pi:G\to H$, such that the restriction of $\pi$ to $H$ is the identity map; this corresponds to a split homomorphism.
Every free factor of a group admits a retraction map.
The \emph{rank} of a free group $F$, written $\rk(F)$, is the cardinality of a minimal generating set for $F$; note that proper free factors of $F$ have strictly smaller rank.

\begin{lemma}
\label{lem:inj-to-non-inj}
There exists an algorithm with input an endomorphism $\phi:F\to F$ and output a free factor $F'$ of $F$,
with retraction map $\pi:F\to F'$,
such that both the maps
\[
\phi|_{F'}\colon F'\to F\quad\&\quad
\overline{\phi}:=\pi\phi|_{F'}\colon F'\to F'
\]
are injective,
and such that for all $n\in\mathbb{N}_0$, the endomorphism $\overline{\phi}^n\pi:F\to F$ is precisely the endomorphism $\pi\phi^n:F\to F$.
\end{lemma}

\begin{proof}
The algorithm begins by determining if $\phi:F\to F$ is injective.
If it is injective, then output $F':=F$, and the result trivially holds.
If it is non-injective, let $\{x_1, \ldots, x_n\}$ be a free basis for $F$ and $\phi(x_i)=w_i$.
The set $\{w_1, \ldots, w_n\}$ Nielsen reduces to a set $\{v_1, \ldots, v_k, 1, \ldots, 1\}$; thus, there is a computable automorphism $\alpha$, the product of the corresponding Nielsen automorphisms, so that $\phi(\alpha(x_i))=v_i$ for $1\leq i\leq k$ and $\phi(\alpha(x_i))=1$ for $i>k$.
Let $F_1$ be the free factor of $F$ generated by $\alpha(x_1), \ldots, \alpha(x_k)$.
By construction, $\phi|_{F_1}:F_1\to F$ is injective, and $\rk(F_1)<\rk(F)$.

Hence, by considering $\pi\phi|_{F_1}:F_1\to F_1$ we can proceed inductively to obtain a free factor $F_{j}$ of $F_{j-1}$, and hence also of $F$, such that $\phi|_{F_j}:F_j\to F$ is injective.
This procedure will terminate when $\pi\phi|_{F_j}$ is injective, and as $\rk(F_{j})<\rk(F_{j-1})$ this will happen after at most $\rk(F)$ iterations.
Set $F':=F_j$ to be the factor which we terminate at.

It is sufficient to prove that $\overline{\phi}^n\pi:F\to F$ and $\pi\phi^n:F\to F$ are the same endomorphism.
This holds because $\phi=\phi\pi$, so
\[
\pi\phi^n=\pi(\phi\pi)^n=(\pi\phi)^n\pi=\overline{\phi^n}\pi
\]
as required.
We can therefore output $F'$, as it has the required properties.
\end{proof}

\subsection{Brinkmann's conjugacy algorithm}
We start by extending our direct generalisation of Brinkmann's conjugacy algorithm.
This result does not follow from the two-sided generalisation, which is Proposition \ref{prop:BCA2} below, but is applied in the proposition's proof.

\begin{theorem}
\label{thm:BCA1}
There exists an algorithm with input an endomorphism $\phi$ of a free group $F$ and two elements $u, v\in F$ and outputs either a pair $(p, k)\in \mathbb{N}_0\times\ker\phi$ such that $\phi^p(u)\sim vk$, or $\mathrm{no}$ if no such pair exists.
\end{theorem}

\begin{proof}
Construct the free factor $F'$, and associated monomorphism $\overline{\phi}:F'\to F'$.
We shall prove that $\phi^p(u)\sim vk$ if and only if $\overline{\phi}^p\pi(u)\sim \pi(v)$; as we can apply Lemma \ref{lem:p-equals-0} to determine if this second problem has a solution, this proves the result.
Here we use that $\overline{\phi}^n\pi=\pi\phi^n:F\to F$, by Lemma \ref{lem:inj-to-non-inj}.

Suppose $\phi^p(u)\sim vk$.
Then we have $\overline{\phi}^p\pi(u)=\pi\phi^p(u)\sim \pi(vk)=\pi(v)$, as required.

Suppose now that $\overline{\phi}^p\pi(u)\sim \pi(v)$, which is equivalent to $\pi\phi^p(u)\sim \pi(v)$.
Both $\pi\phi^p(u)$ and $\pi(v)$ are contained in the free factor $F'$, so by malnormality of $F'$ these elements are conjugate in $F'$, i.e. there exists $w\in F'$ such that $\pi\phi^p(u)= w^{-1}\pi(v)w$.
As $\pi$ is a retraction map this gives $\pi\phi^p(u) = \pi(w^{-1}vw)$, which implies that there exists some $k\in\ker\pi=\ker\phi$ such that $\phi^p(u) \sim vk$, as required.
\end{proof}

Next we have our two-sided generalisation of Brinkmann's conjugacy algorithm.
This has been given the rank of Proposition rather than Theorem, as Theorem \ref{thm:UltraTwCP} from the introduction is simply the combination of this result with Proposition \ref{prop:PowSuperTwCP}.
Note that here, unlike in Theorem \ref{thm:BCA1}, we do not have have an annoying ``$k\in\ker\phi$'' in our identity.
Despite this, the proof is essentially identical.

\begin{proposition}
\label{prop:BCA2}
There exists an algorithm with input an endomorphism $\phi$ of a free group $F$ and two elements $u, v\in F$, and output either a pair $(p, q)\in \mathbb{N}_0\times\mathbb{N}_0$
such that $\phi^p(u)\sim\phi^q(v)$, or $\mathrm{no}$ if no such pair exists.
\end{proposition}

\begin{proof}
Construct the free factor $F'$, and associated monomorphism $\overline{\phi}:F'\to F'$.
We shall prove that there exists a pair
$(p, q)\in \mathbb{N}_0\times\mathbb{N}_0$
such that $\phi^p(u)\sim\phi^q(v)$
if and only if
there exists a pair
$(p', q')\in \mathbb{N}_0\times\mathbb{N}_0$
such that $\overline{\phi}^{p'}\pi(u)\sim\overline{\phi}^{q'}\pi(v)$; as we can apply Proposition \ref{prop:PowSuperTwCP} to determine if this second problem has a solution, this proves the result.
Here we use that $\overline{\phi}^n\pi=\pi\phi^n:F\to F$, by Lemma \ref{lem:inj-to-non-inj}.

Suppose $\phi^p(u)\sim\phi^q(v)$.
Then we have $\overline{\phi}^p\pi(u)=\pi\phi^p(u)\sim\pi\phi^q(v)=\overline{\phi}^q\pi(v)$, as required.

Suppose now that $\overline{\phi}^{p'}\pi(u)\sim \overline{\phi}^{q'}\pi(v)$, which is equivalent to $\pi\phi^{p'}(u)\sim \pi\phi^{q'}(v)$.
Both $\pi\phi^{p'}(u)$ and $\pi\phi^{q'}(v)$ are contained in the free factor $F'$, so by malnormality of $F'$ these elements are conjugate in $F'$, i.e. there exists $w\in F'$ such that $\pi\phi^{p'}(u)= w^{-1}\pi\phi^{q'}(v)w$.
As $\pi$ is a retraction map this gives $\pi\phi^{p'}(u)= \pi(w^{-1}\phi^{q'}(v)w)$, which implies that there exists some $k\in\ker\pi=\ker\phi$ such that $\phi^{p'}(u)\sim \phi^{q'}(v)k$.
Applying $\phi$ to both sides kills the $k$ and so gives the result.
\end{proof}

\subsection{Proof of Theorem \ref{thm:UltraTwCP}}
We now prove Theorem \ref{thm:UltraTwCP} from the introduction.

\begin{theorem}[Theorem \ref{thm:UltraTwCP}]
\label{thm:UltraTwCPBODY}
There exists an algorithm with input an endomorphism $\phi$ of a free group $F$, a natural number $n\in\mathbb{N}_0$ and two elements $u, v\in F$, and output either a pair $(p, q)\in \mathbb{N}_0\times\mathbb{N}_0$
such that $\phi^p(u)$ is $\phi^n$-twisted conjugate to $\phi^q(v)$, or $\mathrm{no}$ if no such pair exists.
\end{theorem}

\begin{proof}
If $n=0$ then the result holds by Proposition \ref{prop:BCA2}.
If $n>0$ then
by Proposition \ref{prop:ExtendedTwCP}, we can determine if there exists a pair $(p', q')\in \mathbb{N}_0\times\mathbb{N}_0$ such that $\phi^{p'}(u)$ is $\phi^n$-twisted conjugate to $\phi^{q'}(v)$.
\end{proof}

\subsection{Brinkmann's equality algorithm}
In his paper, Brinkmann first proves his conjugacy algorithm, and then applies a standard trick to prove the equality algorithm \cite[Section 6]{Brinkmann2010Detecting}.
We apply the same trick to Theorem \ref{thm:BCA1} and Proposition \ref{prop:BCA2}.

\begin{theorem}
\label{thm:BEA1}
There exists an algorithm with input an endomorphism $\phi$ of a free group $F$ and two elements $u, v\in F$, and output either a pair $(q, k)\in \mathbb{N}_0\times\ker\phi$ such that $\phi^p(u)=vk$, or $\mathrm{no}$ if no such exponent exists.
\end{theorem}

\begin{proof}
Let $F_a := F\ast \langle a\rangle$ and define $\varphi:F_a\to F_a$ by letting $\varphi(x) = \phi(x)$ if $x \in F$ and $\varphi(a) = a$.
If $w \in F$, then $wa$ is cyclically reduced in $F_a$, so that $\phi^p(u)=vk$ if and only if $\varphi^p(ua)\sim vka$.
Now Theorem \ref{thm:BCA1} completes the proof.
\end{proof}

Removing conjugacy from Theorem \ref{thm:UltraTwCP} means we are in the setting of Proposition \ref{prop:BCA2}.

\begin{theorem}[Theorem \ref{thm:equality}]
\label{thm:equalityBODY}
There exists an algorithm with input an endomorphism $\phi$ of a free group $F$ and two elements $u, v\in F$, and output either a pair $(p, q)\in \mathbb{N}_0\times\mathbb{N}_0$
such that $\phi^p(u)=\phi^q(v)$, or $\mathrm{no}$ if no such pair exists.
\end{theorem}

\begin{proof}
Define $\varphi:F_a\to F_a$ as in the proof of Theorem \ref{thm:BEA1}.
If $w \in F$, then $wa$ is cyclically reduced in $F_a$, so that $\phi^p(u) = \phi^q(v)$ if and only if $\varphi^p(ua)\sim \varphi^q(va)$.
Now Proposition \ref{prop:BCA2} completes the proof.
\end{proof}


\section{Ascending HNN-extensions of free groups}
\label{sec:MainProofs}
We now prove our main result, Theorem \ref{thm:CP}, which states that ascending HNN-extensions of free groups have decidable Conjugacy Problem.
In free-by-cyclic groups $F\rtimes\langle t\rangle$, elements have the form $xt^k$ for $x\in F$ and $k\in\mathbb{Z}$. For ascending HNN-extensions our general form is slightly more complicated, with elements of $G=\langle F, t\mid t^{-1}xt=\phi(x), x\in F\rangle$ having the form $t^{i}xt^{-j}$ with $x\in F$, $i, j\in\mathbb{N}_0$. 

\begin{theorem}[Theorem \ref{thm:CP}]
\label{thm:CPBODY}
There exists an algorithm to solve the Conjugacy Problem in ascending HNN-extensions of finitely generated free groups.
\end{theorem}

\begin{proof}
Let $F\ast_{\phi}$ be the ascending HNN-extension of the finitely generated free group $F$ induced by the injective endomorphism $\phi:F\to F$.
Let $g'=t^{m_0}ut^{-m_1}$ and $h'=t^{n_0}vt^{-n_1}$, $m_0, m_1, n_0, n_1\in\mathbb{N}_0$ and $u, v\in F$, be our input elements.
The first step of the algorithm is to cyclically shift these elements so that they are of the form
$g=ut^{-m}$ and $h=vt^{-n}$ with $m, n\in\mathbb{Z}$ and $u, v\in F$ (so $m:=m_1-m_0$ and $n:=n_1-n_0$).
By considering the retraction map $F\ast_{\phi}\twoheadrightarrow\langle t\rangle$, which has $F$ in the kernel, we see that if $g$ and $h$ are conjugate then $m=n$.
Therefore, the second step of the algorithm is to compute $m$ and $n$, and if they are non-equal output that the elements are non-conjugate.
We now assume that $g=ut^{-n}$ and $h=vt^{-n}$, and furthermore that $n\geq0$, as if $n<0$ we consider $u^{-1}g^{-1}u$ and $v^{-1}h^{-1}v$ instead.

Then $g$ and $h$ are conjugate if and only if there exists some triple $(x, p, q)\in F\times\mathbb{N}_0\times\mathbb{N}_0$ such that
$g=(t^qxt^{-p})^{-1}\cdot h\cdot t^qxt^{-p}$.
We then have the following sequence of equivalent identities.
\begin{align*}
ut^{-n}&=(t^qxt^{-p})^{-1}\cdot vt^{-n}\cdot t^qxt^{-p}\\
t^{-p}ut^{p-n}&=x^{-1}t^{-q}vt^{q-n}x\\
\phi^p(u)t^{-n}&=x^{-1}\phi^{q}(v)t^{-n}x\\
\phi^p(u)t^{-n}&=x^{-1}\phi^{q}(v)\phi^n(x)t^{-n}\\
\phi^p(u)&=x^{-1}\phi^{q}(v)\phi^n(x)
\end{align*}
Therefore, to test that $g$ and $h$ are conjugate it is sufficient to input the endomorphism $\phi$, the number $n\in\mathbb{N}_0$ and the elements $u, v\in F$ into the algorithm of Proposition \ref{prop:UltraTwCPINJECTIVE},
with the elements being conjugate in $F\ast_{\phi}$ if and only if a pair $(p, q)$ is output.
\end{proof}


\section{On Dru\c{t}u and Sapir's one-relator group}
\label{sec:oneRel}

Consider the one-relator group $G:=\langle a, t\mid t^2at^{-2}a^{-2}\rangle$, which was proven by Dru\c{t}u and Sapir to be residually finite but not linear \cite[Theorem 4]{Drutu2005Nonlinear}.
We now prove that the standard routes to resolving the Conjugacy Problem, as well as previous results on the Conjugacy Problem for one-relator groups, as cited in the introduction, do not apply to this group.
Once the relevant definitions are given, it is clear that the results cited in the introduction are not immediately applicable to the presentation $\langle a, t\mid t^2at^{-2}a^{-2}\rangle$; we go further and check that no one-relator presentation of $G$ is covered by these results.
However, this group is an ascending HNN-extensions of $F(a, b)$, with $G\cong\langle a, b, t\mid tat^{-1}=b, tbt^{-1}=a^2\rangle$, and hence $G$ has decidable Conjugacy Problem by Theorem \ref{thm:CP}.

\p{Semi-hyperbolicity}
Hyperbolic, biautomatic and CAT(0) groups have decidable Conjugacy Problem.
One reason for this is that these groups are all semi-hyperbolic \cite[Examples 2.1--2.3]{Alonso1995semihyperbolic}, and semi-hyperbolic groups have decidable Conjugacy Problem \cite[Theorem 5.1]{Alonso1995semihyperbolic}.
We therefore prove that the group $G$ is not semi-hyperbolic, which eliminates this standard route to resolving the Conjugacy Problem.
The \emph{Baumslag--Solitar group $\BS(m, n)$} is the group with presentation $\langle x, y\mid y^{-1}x^my=x^n\rangle$, $m, n\neq0$, and containing some $\BS(1, n)$ with $|n|>1$ as a subgroup is an obstruction to being semi-hyperbolic \cite[Proposition 7.17]{Alonso1995semihyperbolic}.

\begin{lemma}
\label{lem:hyperbolic}
$G$ contains a subgroup isomorphic to $\BS(1, 2)$, and hence is not semi-hyperbolic.
\end{lemma}

\begin{proof}
Consider the presentation $\mathcal{P}:=\langle a, s, t \mid sas^{-1} = a^2, s = t^2 \rangle$.
This represents a free product of $\BS(1, 2)=\langle a, s\rangle$ and $\mathbb{Z}=\langle t\rangle$, amalgamated along $s=t^2$, and so $\BS(1, 2)$ embeds into the group defined by $\mathcal{P}$.
Applying to $\mathcal{P}$ the Tietze transformation which removes the generator $s$ gives the original presentation for $G$, and so $\mathcal{P}$ defines $G$.
It follows that the subgroup $\langle a, t^2\rangle$ of $G$ is isomorphic to $\BS(1, 2)$.
\end{proof}

As one-relator groups with torsion are hyperbolic, the above means that Newman's result \cite{Newman1968some} does not apply to $G$.
The fact that the word $t^2at^{-2}a^{-2}$ is not a proper power of any element in $F(a, t)$ gives a different, more combinatorial obstruction to the use of Newman's result here \cite{Karrass1960elements}.

\p{Generalised Baumslag--Solitar groups}
A \emph{generalised Baumslag--Solitar (GBS) group} is a group which decomposes as a graph of groups with all edge and vertex groups infinite cyclic.
GBS groups have decidable Conjugacy Problem \cite{Lockhart1992conjugacy}.
We therefore prove that the group $G$ is not a GBS group, which eliminates this route to resolving the Conjugacy Problem.
Firstly, we re-phrase a result of Levitt \cite[Corollary 7.7]{Levitt2015Quotients}:

\begin{proposition}[Levitt]
	\label{prop:Levitt}
	A GBS group is linear if and only if it is residually finite.
\end{proposition}

\begin{proof}
	A GBS group is residually finite if and only if it is either soluble or virtually $F\times\mathbb{Z}$ for $F$ free \cite[Corollary 7.7]{Levitt2015Quotients}.
	Any $F\times\mathbb{Z}$ group is linear, and hence so is any virtually $F\times\mathbb{Z}$ group.
	A GBS group is soluble if and only if it is isomorphic to $\BS(1, n)$ for some $|n|>0$, and again these groups are linear.
	The result follows immediately.
\end{proof}

\begin{lemma}
	$G$ is not a GBS group
\end{lemma}

\begin{proof}
	The group $G$ is residually finite but not linear \cite[Theorem 4]{Drutu2005Nonlinear}, so the result follows by Proposition \ref{prop:Levitt}.
\end{proof}

\p{Dehn functions}
HNN-extensions of free groups with Dehn function less than $n^2 \log n$ (with a somewhat technical definition of ``less than'') have decidable Conjugacy Problem \cite[Theorem 2.5]{olshanskii2006SmallDehn}.
As noted above, $G$ is an HNN-extension of a free group, so we now eliminate this Dehn function-based route to resolving the Conjugacy Problem as follows:

\begin{lemma}
$G$ has at least exponential Dehn function.
\end{lemma}

\begin{proof}
One-relator groups have cohomological dimension $2$, and hence Gersten's Theorem tells us that subgroups give lower bounds on the Dehn function \cite[Theorem C]{Gersten1992Dehn}.
Therefore, by Lemma \ref{lem:hyperbolic}, the Dehn function of $\BS(1, 2)$ gives a lower bound for the Dehn function of $G$.
The result follows as $\BS(1, 2)$ has exponential Dehn function \cite{Groves2001Isoperimetric}.
\end{proof}

\p{Cyclically and conjugacy pinched}
A group is a \emph{cyclically pinched one-relator group} if it admits a decomposition as a free product of two free groups amalgamated over a non-trivial cyclic subgroup.
Such a decomposition corresponds to a presentation $\langle \mathbf{x}, \mathbf{y}\mid U=V\rangle$ with $|\mathbf{x}|, |\mathbf{y}|\geq1$ and $U \in F(\mathbf{x})$, $V\in F(\mathbf{y})$.
Cyclically pinched groups have decidable Conjugacy Problem.

\begin{lemma}
$G$ is not a cyclically pinched one-relator group.
\end{lemma}

\begin{proof}
A non-free $n$-generator one-relator group cannot be generated by fewer than $n$-elements \cite[II.5.11]{LyndonSchupp2001}.
Therefore, as $G$ is two-generated, if it admits a presentation $\langle \mathbf{x}, \mathbf{y}\mid U=V\rangle$ with $|\mathbf{x}|, |\mathbf{y}|\geq1$ then we must have $|\mathbf{x}|=1=|\mathbf{y}|$, and so $G$ must then admit a presentation $\mathcal{P}:=\langle x, y\mid x^m=y^n\rangle$.
By applying Brown's algorithm, we see that $G$ is free-by-cyclic \cite[Sections 3 \& 4]{Brown1987trees} (see also \cite[Theorem 2(1)]{Baumslag2009Virtual}).
However, by Lemma \ref{lem:hyperbolic}, $G$ contains a subgroup $H$ isomorphic to $\BS(1, 2)$, and the derived subgroup of $H$ is a non-finitely generated abelian group, and so $G$ contains no free normal subgroup $F$ with $G/F$ abelian.
Hence, $G$ is not a cyclically pinched one-relator group.
\end{proof}

A group is a \emph{conjugacy pinched one-relator group} if it admits a decomposition as an HNN-extension of a free group with cyclic associated subgroups.
Such a decomposition corresponds to a presentation $\langle \mathbf{x}, y\mid y^{-1}Uy=V\rangle$ with $U, V\in F(\mathbf{x})$ non-empty.
Conjugacy pinched groups have decidable Conjugacy Problem.

\begin{lemma}
$G$ is not a conjugacy pinched one-relator group.
\end{lemma}

\begin{proof}
A non-free $n$-generator one-relator group cannot be generated by fewer than $n$-elements \cite[II.5.11]{LyndonSchupp2001}.
Therefore, as $G$ is two-generated, if it admits a presentation $\langle \mathbf{x}, y\mid y^{-1}Uy=V\rangle$ with $U, V \in F(\mathbf{x})$ then we must have $|\mathbf{x}|=1$, and so $G$ must then admit a presentation $\mathcal{P}:=\langle x, y\mid y^{-1}x^my=x^n\rangle$.
This is a presentation of the Baumslag--Solitar group $\BS(m, n)$.
However, a Baumslag--Solitar group is linear if and only if it is residually finite \cite{Meskin1972nonresidually} \cite{Volvachev1985linear}, but $G$ is residually finite and non-linear.
Hence, $G$ is not a conjugacy-pinched one-relator group.
\end{proof}

\p{Pride's unique max-min condition}
For a word $R\in F(\mathbf{x})$, write $R^{(i)}$ for the prefix of $R$ of length $i$.
A one-relator presentation $\mathcal{P}=\langle \mathbf{x}\mid R\rangle$ satisfies \emph{Pride's unique max-min condition} if there exists a homomorphism $\rho:F(\mathbf{x})\onto\mathbb{Z}$ satisfying the following:
\begin{enumerate}
\item $R\in\ker\rho$.
\item $\mathbf{x}\cap\ker\rho=\emptyset$.
\item There exists $0<i, j\leq |R|$ such that for all $0<k\leq |R|$ with $i\neq k\neq j$ we have $\rho(R^{(i)})>\rho(R^{(k)})>\rho(R^{(j)})$.
\end{enumerate}
The group defined by such a presentation $\mathcal{P}$ has decidable Conjugacy Problem.

\begin{lemma}
No one-relator presentation for $G$ satisfies Pride's unique max-min condition.
\end{lemma}

\begin{proof}
We start by considering free-by-cyclic groups.
In such a group, no element is conjugate to a proper power of itself, which is a straightforward exercise applying the fact that elements of free groups have only finitely many roots.

In the group $G$, the element $a\in G$ is conjugate to a proper power of itself, and so $G$ does not embed into any free-by-cyclic group.
On the other hand, if a one-relator presentation $\mathcal{P}$ satisfies Pride's unique max-min condition, then the group defined by $\mathcal{P}$ embeds into a free-by-cyclic group $F\rtimes\mathbb{Z}$ \cite[Section 5]{Pride2008residual}.
The result now follows.
\end{proof}

\p{Conjugacy Separability}
We finish this section by discussing conjugacy separability.
A group $H$ is \emph{conjugacy separable} if for every pair of non-conjugate elements $g, h\in H$ there exists a homomorphism $\psi: H\to K$ with $K$ finite and such that $\psi(g)$ and $\psi(h)$ are non-conjugate in $K$.
Finitely presented conjugacy separable groups have decidable Conjugacy Problem, and are residually finite.
Therefore, proving conjugacy separability is a standard, although generally difficult, route to proving the decidability of the Conjugacy Problem for finitely presented residually finite groups.
Indeed, it is an open problem whether every free-by-cyclic group is conjugacy separable \cite[Problem 19.41]{khukhro2014unsolved}.

So far as we know, the conjugacy separability of the Dru\c{t}u--Sapir group $G$ is open.
Indeed, we could try and apply the standard result of Ribes, Segal and Zalesskii \cite{Ribes1998ConjugacySeparability}.
To apply this result, our group $G$ must decompose as a free product of conjugacy separable groups with cyclic associated subgroup; this is indeed the case here, with $G\cong\langle a, s, t\mid sas^{-1}=a^2, s=t^2\rangle$ factoring as a free product with amalgamation of $\operatorname{BS}(1, 2)=\langle a, s\rangle$ and $\mathbb{Z}=\langle t\rangle$ over $s=t^2$, as observed in the proof of Lemma \ref{lem:hyperbolic}, and furthermore $\operatorname{BS}(1, 2)$ is conjugacy separable \cite[Theorem 10]{Moldavanskii2018Residual}.
Ribes--Segal--Zalesskii then require both factor groups to satisfy certain properties regarding their profinite completions, including being ``cyclic subgroup separable''. Although $\mathbb{Z}$ easily satisfies these properties, $\operatorname{BS}(1, 2)$ is not cyclic subgroup separable; a quick proof is given in \cite[Section 4]{Moldavanskii2018Residual}, while \cite[Theorem 1.1]{Lopez2025Separability} gives the complete picture of cyclic subgroup separability for GBS groups.

\bibliographystyle{amsalpha}
\bibliography{BibTexBibliography}
\end{document}